\documentclass[reqno,11.5pt]{article}

\usepackage{diagrams}
\usepackage{amssymb}
\usepackage{tikz}
\usepackage{amsmath}
\usepackage{latexsym}
\usepackage{mathrsfs}
\usepackage{amsthm}
\usepackage{verbatim}
\usepackage{stmaryrd}
\usepackage{graphicx}
\usepackage{epstopdf}
\usepackage{epsfig}
\diagramstyle[labelstyle=\scriptstyle]
\usepackage{color}
\usepackage{subfigure}
\usepackage{float}
\usepackage[algo2e,linesnumbered,ruled]{algorithm2e}
\usepackage{titlesec}
\usepackage{bm}
\usepackage[colorlinks,linkcolor=blue,anchorcolor=green,citecolor=red]{hyperref}
\usepackage{amsfonts}
\usepackage{fancyhdr}
\usepackage{amscd}
\usepackage{cases}
\usepackage{tikz-cd}

\newtheorem{theorem}{Theorem}
\newtheorem{remark}{Remark}

\newtheorem{lemma}{Lemma}

\newcommand{\0}{\mathaccent23}

\newcommand\tr{\operatorname{tr}}

\newcommand\grad{\operatorname{grad}}
\renewcommand\div{\operatorname{div}}
\newcommand\curl{\operatorname{curl}}

\newcommand\x{\times}

\newcommand{\bs}{{\scriptscriptstyle \bullet}}

\makeatletter
\@addtoreset{equation}{section}
\makeatother

\usepackage{geometry}
\begin{document}

\title{Nonstandard finite element de Rham complexes \\ on cubical meshes}

\author{Andrew Gillette \thanks{Department of Mathematics, University of Arizona, 617 N. Santa Rita Ave.,
Tucson, Arizona, USA. email: {\tt agillette@math.arizona.edu}} \and Kaibo Hu \thanks{Department of Mathematics, University of Oslo, PO Box 1053 Blindern, NO 0316 Oslo, Norway.  email: {\tt kaibohu@math.uio.no}} \and Shuo Zhang \thanks{LSEC, Institute of Computational Mathematics and Scientific/Engineering Computing, Academy of Mathematics and System Sciences, and National Centre for Mathematics and Interdisciplinary Sciences, Chinese Academy of Sciences, Beijing 100190, People's Republic of China {\tt email: szhang@lsec.cc.ac.cn}}}

\date{}

\maketitle

\begin{abstract}
We propose two general operations on finite element differential complexes on cubical meshes that can be used to construct and analyze sequences of ``nonstandard'' convergent methods.
The first operation, called DoF-transfer, moves edge degrees of freedom to vertices in a way that reduces global degrees of freedom while increasing continuity order at vertices.
The second operation, called serendipity, eliminates interior bubble functions and degrees of freedom locally on each element without affecting edge degrees of freedom.
These operations can be used independently or in tandem to create nonstandard complexes that incorporate Hermite, Adini and ``trimmed-Adini'' elements.
We show that the resulting elements provide convergent, non-conforming methods for problems requiring stronger regularity and satisfy a discrete Korn inequality.
We discuss potential benefits of applying these elements to Stokes, biharmonic and elasticity problems.
\end{abstract}

\section{Introduction}

To obtain compatible and stable numerical  schemes in the framework of discrete differential forms and  finite element exterior calculus \cite{Arnold.D;Falk.R;Winther.R.2006a, Arnold.D;Falk.R;Winther.R.2010a,Hiptmair.R.2002a},  physical variables are discretized in discrete differential complexes.  There have been many successful finite element de Rham complexes, consisting of the Lagrange, 1st and 2nd N\'{e}d\'{e}lec, Raviart-Thomas, Brezzi-Douglas-Marini (BDM) and discontinuous elements. On simplicial meshes, the construction of these elements is classical \cite{Nedelec.J.1986a, Nedelec.J.1980a, Raviart.P;Thomas.J.1977a, brezzi1985two}. On cubical meshes, the analogue of the Raviart-Thomas family has a tensor product nature \cite{Boffi.D;Brezzi.F;Fortin.M.2013a}, while the analogue of the BDM families, denoted as $\mathcal{S}_{r}\Lambda^{k}$, are developed in \cite{arnold2014finite} as a resolution of the serendipity element (c.f. \cite{arnold2011serendipity}). 
These results are summarized in the Periodic Table of the Finite Elements~\cite{arnold2014periodic}. The recently defined trimmed serendipity family, denoted $\mathcal{S}_r^-\Lambda^k$, provides a distinct resolution of the serendipity element as subcomplex of the de Rham sequence with some computational advantages~\cite{GK2016}.

While the de Rham complex describes regularity spaces appropriate for Hodge-Laplacian problems, it is not suited for flow problems that involve the same differential operators but require stronger regularities. 
A prominent example is the Stokes problem and the corresponding ``Stokes complex'', given by
$$
\begin{diagram}
0 & \rTo & \mathbb{R}  & \rTo & H^{2}(\Omega)& \rTo^{\curl} & \left [H^{1}(\Omega)\right ]^{2} & \rTo^{\div} & L^{2} & \rTo &0,
\end{diagram}
$$
in two space dimensions. The importance of precisely preserving mass conservation (also called the divergence-free condition) has been shown in \cite{john2017divergence}. 
Accordingly, finite dimensional subcomplexes of the Stokes complexes can be used to construct suitable elements.
Falk and Neilan \cite{falk2013stokes} construct a sequence starting with the $C^{1}$ Argyris element on 2D triangular meshes. Neilan and Sap \cite{neilan2016stokes}  give an analogous construction on cubical meshes, which is a resolution of the Bogner-Fox-Schimt (BFS) $C^{1}$ element. Tensor product constructions based on splines in the context of the isogeometric  analysis are proposed in \cite{buffa2011isogeometric}.  We also refer to \cite{neilan2015discrete,neilan2016stokes} and the reference therein for relevant work in 3D and refer to \cite{zhang2005new,arnold1992quadratic,christiansen2016generalized} for constructions on macroelements.
To obtain these finite element complexes with higher regularity, one usually needs a high polynomial degree and extra derivative degrees of freedom, or certain mesh conditions.

An attractive approach to constructing Stokes elements in a simple fashion while ensuring mass conservation is to use non-conforming elements. 
For instance, the discrete Stokes complex
\begin{align}\label{seq: non-conforming }
\begin{diagram}
0 & \rTo & \mathbb{R}  & \rTo & \mathrm{Morley} & \rTo^{\curl_{h}} & \left [\mathrm{Crouzeix-Raviart}\right ]^{2} & \rTo^{\div_{h}} & \mathrm{DG} & \rTo &0,
\end{diagram}
\end{align}
consists of the  Morley and the vector Crouzeix-Raviart elements and piecewise derivatives $\curl_{h}$ and $\div_{h}$. Although the Morley element is an elegant non-conforming choice for the biharmonic problem, it is not convergent for the Poisson equation (c.f. \cite{Wang.M2001pt,nilssen2001robust}). Correspondingly, the vector Crouzeix-Raviart element is not desirable as an  $H(\div)$ element for the Hodge Laplacian problem. For the Darcy-Stokes-Brinkman problem, which is a perturbation of the flow problem with an $H(\div)$ Hodge Laplacian, the numerical experiments in \cite{mardal2002robust} demonstrate that the convergence with the Crouzeix-Raviart element deteriorates as the coefficient in front of the Laplace operator tends to zero. As argued in \cite{mardal2002robust,tai2006discrete}, a possible cure for this problem is to seek elements that are conforming in $H(\div)$ and non-conforming convergent in $\left [H^{1}\right ]^{n}$, then further construct discrete sequences that are conforming to the de Rham complex and non-conforming convergent as a Stokes complex. A similar complex is constructed on quadrilateral grids \cite{Zhang.S2016}.

On a given mesh, any approximation space can be characterized by local functions on each cell and the continuity across the boundaries of the cells, which we shall call the ``inter-element continuity'' in subsequent discussions.
Finite elements are a special kind of approximation in which  inter-element continuity can be ensured via constraints local to each element. Precisely, in Ciarlet's classical definition of finite elements,  a finite element is denoted  as a triple $(P,G,D)$ where $P$ is the approximation space (in this case, a space of polynomial differential forms), $G$ is the geometry (in this case, an $n$-cube) and $D$ is the set of functionals on $P$ that are associated to portions of $G$ (e.g.\ evaluation at a specific vertex of $G$, or integration against a specific test function along an edge of $G$). The inter-element continuity is imposed by requiring that these DoFs are single valued on neighboring elements. This is an extra locality condition and in general, approximation spaces, such as some spline spaces, may not fulfill this condition. In the theory of finite element systems (FES) \cite{christiansen2011topics},  another framework of finite elements developed by Christiansen and collaborators,  no degrees of freedom are explicitly used. The approximation space is viewed as a collection of spaces living on all subcells (of all dimensions) of the mesh. The inter-element continuity is imposed by requiring that all the traces of functions on higher dimensional cells coincide with the functions on lower dimensional cells (therefore the global functions are single valued). The original theory of the FES is tailored for finite element differential forms \cite{christiansen2011topics}, and is generalized to spaces with higher continuity in \cite{christiansen2016generalized}.

In general, there are two ways to reduce a piecewise-defined approximation space: reduce the size of local shape function spaces or increase the inter-element continuity. These two operations can always be done. However, from the perspective of finite elements, we require the first operation (local reduction) to preserve certain local approximation properties (e.g. containment of polynomials of a certain degree) and require the second (global reduction) to retain the locality of the approximation (e.g. finite elements defined in Ciarlet's sense or in the FES sense). From a viewpoint of approximation theories, the locality of the approximation spaces plus the local approximation properties guarantees the global approximation properties. 
In light of the finite element exterior calculus, one further demands that these two operations also preserve the exactness of a sequence of finite element spaces.

In this paper, we provide a canonical procedure for modifying conforming deRham subcomplexes into  non-conforming  convergent sequences of spaces of elements that can be used in problems requiring greater regularity. In 2D, our construction leads to simple elements for the Darcy-Stokes-Brinkman problem while in 3D our construction yields the Adini plate element~\cite{adini} and an $H(\curl)$ element that is conforming for Poisson/Stokes type problems. Instead of enriching an existing space with bubble functions as in, e.g.~\cite{mardal2002robust,tai2006discrete}, we reduce the local and global degrees of freedom of standard de Rham finite element complexes to sequences that provide the desired properties. From the perspective of non-conforming methods, consistency is preserved in the space reduction.  More precisely, if the consistency error in the Strang lemma converges to zero as $h\rightarrow 0$ for a space $V_{h}$, then the consistency error for any subspace of $V_{h}$ also tends to zero.

The ideas presented here are related to recent investigations into ``nonstandard'' vector elements. Stenberg introduced a family of ``nonstandard elements'' for the $H(\div)$ space on triangles and tetrahedra in~\cite{Stenberg2010}. These elements are subspaces of the classical BDM spaces that are continuous at vertices of the mesh. Enforcing this continuity constraint reduces the number of global degrees of freedom (DoFs). The canonical interpolations defined by the DoFs, however, do not commute with the differential operations; instead, Stenberg uses macroelement techniques to show the inf-sup condition holds~\cite{Stenberg2010}.  In~\cite{christiansen2016nodal}, Christiansen, Hu, and Hu derive complexes of nonstandard elements from a different perspective. The elements are constructed as variations of copies of Lagrange and Hermite elements, and as a result, canonical nodal bases are available in some cases. The work is partly motivated by an element for linear elasticity~\cite{hu2015family,Hu2015}. The extended periodic table in \cite{christiansen2016nodal} includes the triangular and tetrahedral Stenberg elements as a special case for the $H(\div)$ space and several new $H(\curl)$ finite elements.

The results presented in this paper can be regarded as a continuation of the discussions in \cite{christiansen2016nodal}. In both the simplicial and the cubical cases, the number of global DoFs is reduced. However, there are some prominent differences. The triangular Hermite element does not lead to convergent schemes for biharmonic problems, although modifications on the shape functions are possible to make it convergent \cite{Wang.M;Shi.Z;Xu.J2007NM}. 
In contrast, cubical elements with vertex derivative DoFs usually \textit{do} lead to convergent schemes for the biharmonic equation due to the geometric symmetry. The Adini plate element serves as an illuminating example: it is globally $C^{0}$ with derivative DoFs at vertices and comparable to the Hermite element on triangular meshes (c.f. \cite{hu2016capacity} and the references therein). In this paper, we will construct the Adini element by a new procedure, and place it within an entire finite element complex on squares or cubes.

Our approach to element construction is to provide two operations that can be applied to known finite element complexes: a ``DoF-transfer'' operation and a ``serendipity'' operation.  The DoF-transfer operation moves some edge DoFs to vertex DoFs. The local shape function spaces and bubbles do not change and the DoF-transfer preserves unisolvence and exactness. The serendipity operation eliminates some interior bubbles and DoFs at the same time in a way that preserves unisolvence and exactness. The idea of the serendipity operation bears some similarity to the serendipity reduction process in virtual element methods as described in~\cite{da2016serendipity,da2016serendipity2}, however, virtual element techniques are chiefly concerned with applicability to many-sided polygonal and generic polyhedral element geometries. Combining these two operations, we reduce the local and global spaces of various known discrete complexes on cubical meshes \cite{arnold2014finite,arnold2015finite,GK2016} to obtain Hermite, Adini and trimmed-Adini type families.  {The $H(\div)$ element in the 2D Hermite complex coincides with the nonstandard element introduced in Stenberg \cite{Stenberg2010}. }

The rest of this paper will be organized as follows. In Section \ref{sec:notation}, we introduce notation and background material. In Section \ref{sec:seren-transf}, we define the DoF-transfer and  serendipity operations. In Section \ref{sec:elements}, we define the Hermite, Adini, and trimmed-Adini elements and complexes.  Section \ref{sec:convergence} is devoted to the approximation and convergence properties of the new finite elements, including the discrete Korn inequality and the convergence as  non-conforming  elements.  In Section \ref{sec:minimal}, we further use the idea of the serendipity operation to construct a complex of minimal elements that are conforming as a de Rham complex and non-conforming as a Stokes complex.  We conclude in Section \ref{sec:conclusion} with some remarks and future directions.

\section{Notation and background}\label{sec:notation}

Let $\Omega\subset \mathbb{R}^{n}$ be a polyhedral domain and $\mathcal{T}_{h}$ be a mesh on $\Omega$ consisting of $n$-cubes. Sometimes we use $\mathcal{T}_{h}^{2}$ and $\mathcal{T}_{h}^{3}$ to denote 2D and 3D meshes respectively. 
A single square element is denoted by $\square_2$ and a single cube element by $\square_3$.
Given a mesh, we use $\mathcal{V}$, $\mathcal{E}$, $\mathcal{F}$ and $\mathcal{K}$ to denote the sets of vertices, edges, faces and three dimensional cells (zero to three dimensional cells) respectively. 
 We use $\bm{\nu}_{f}$ and $\bm{\tau}_{f}$ to denote the unit normal and tangential vectors of a cell  $f$, respectively, when applicable.

We use $\mathcal{P}_{r, s}(K)$, or simply $\mathcal{P}_{r, s}$, to denote the polynomial space of degree $r$ in $x$ and degree $s$ in $y$ on a domain $K$, i.e.
$$
\mathcal{P}_{r, s}:=\mathcal{P}_{r}(x)\times \mathcal{P}_{s}(y). 
$$
Similarly, in three dimensions, we define
$$
\mathcal{P}_{r, s, t}:=\mathcal{P}_{r}(x)\times \mathcal{P}_{s}(y)\times \mathcal{P}_{t}(z), \quad 
$$
When the orders are equal, we get the tensor product polynomial spaces, denoted
$$
\mathcal{Q}_{r}:=\mathcal{P}_{r,r},\quad\text{and}\quad\mathcal{Q}_{r}:=\mathcal{P}_{r, r, r}.
$$
%
%We define $\ker \left ( d, V\right )$ as the kernel of the differential operator $d$ in the space $V$, i.e.
%$$
%\ker \left ( d, V\right ):=\{v\in V: dv=0\}.
%$$
%
Following \cite{arnold2014finite}, we define the \textit{superlinear degree} of a monomial to be the total degree of factors that enter with quadratic or higher degrees. For example, the superlinear degree of $x^{2}y^{3}z$ is 5, where we count the degrees of $x$ and $y$, but not $z$. We use $\mathcal{S}_{r}$ to denote the polynomials with superlinear degree less than or equal to $r$.  There is a simple nesting of spaces given by $\mathcal{P}_{r}\subset \mathcal{S}_{r}\subset \mathcal{Q}_{r}$.

We review some basic facts from homological algebra; further details can be found, for instance, in \cite{bott2013differential, dummit2004abstract, Arnold.D;Falk.R;Winther.R.2006a}. A differential complex is a sequence of spaces $V_{i}$ and operators $d_{i}$ such that
\begin{align}\label{general-complex}
\begin{diagram}
0 & \rTo & V_{1} & \rTo^{d_{1}} &V_{2}&\rTo^{d_{2}} & \cdots  &  \rTo^{d_{n-1}} &V_{n} & \rTo^{d_n} & 0,
\end{diagram}
\end{align}
satisfying $d_{i+1}d_{i}=0$ for $i=1, 2, \cdots, n-1$; typically, this condition is denoted as $d^{2}=0$. Let $\ker(d_{i})$ be the kernel space of the operator $d_{i}$ in $V_{i}$, and $\mathrm{Im}(d_{i})$ be the image of the operator $d_{i}$ in $V_{i+1}$. Due to the complex property $d^{2}=0$, we have $\ker(d_{i})\subset \mathrm{Im}(d_{i-1})$ for each $i\geq 2$. Furthermore, if $\ker(d_{i})= \mathrm{Im}(d_{i-1})$, we say that the complex \eqref{general-complex} is exact at $V_{i}$. At the two ends of the sequence, the complex is exact at $V_{1}$ if $d_{1}$ is injective (with trivial kernel), and is exact at $V_{n}$ if $d_{n}$ is surjective (with trivial cokernel).  The complex \eqref{general-complex} is called exact if it is exact at all the spaces $V_{i}$. If each space in \eqref{general-complex} has finite dimension, then a necessary (but not sufficient) condition for the exactness of \eqref{general-complex} is the following dimension condition:
$$
\sum_{i=1}^{n} (-1)^{i}\dim (V_{i})=0. 
$$

For finite element spaces, the exactness of the local shape function space on each mesh element is called \textit{local exactness}, while that of the whole space, with the inter-element continuity taken into account, is called \textit{global exactness}. In subsequent discussions, the local exactness follows from classical results. Therefore we mean global exactness as a default. 

Another complex 
\begin{align*}
\begin{diagram}
0 & \rTo & W_{1} & \rTo^{d_{1}} &W_{2}&\rTo^{d_{2}} & \cdots  &  \rTo^{d_{n-1}} &W_{n} & \rTo^{} & 0
\end{diagram}
\end{align*}
is called a \textit{subcomplex} of \eqref{general-complex} if for each $i=1, 2, \cdots, n$, $W_{i}$ is a subspace of $V_{i}$ and the differential operators are the same.

Define 
$$
H(d, \Omega):=\left \{u\in L^{2}(\Omega): du\in L^{2}(\Omega)\right \},
$$
where $d$ is a differential operator. We use $\|u\|_{s, \Omega}$ to denote the $H^{s}$ norm of $u$ on $\Omega$.
 When there is no possible confusion, we also use $\|\cdot\|_{s}$ instead of $\|\cdot\|_{s, \Omega}$ and use $\|\cdot\|$ instead of the $L^{2}$ norm $\|\cdot\|_{0}$.
There are two de Rham sequences with the above type of Sobolev regularity in 2D:
  \begin{equation}\label{sequence1}
\begin{CD}
0@>>>\mathbb{R}@>>>H(\curl, \Omega) @>\curl>> {H}(\mathrm{div}, \Omega) @>\mathrm{div} >> L^{2}(\Omega)  @ > >>  0,
\end{CD}
\end{equation}
and
 \begin{equation}\label{sequence2}
\begin{CD}
0@>>>\mathbb{R}@>>>H({\grad}, \Omega) @>{\grad}>> {H}(\mathrm{rot}, \Omega) @>\mathrm{rot} >> L^{2}(\Omega)  @ > >>  0.
\end{CD}
\end{equation}
Here, the 2D curl operator maps a scalar function to a vector valued function, defined by 
\begin{align}\label{2Dcurl}
\curl u:=(-\partial_{2}u, \partial_{1}u)^{T}.
\end{align}
One can obtain \eqref{sequence2} by rotating the spaces in \eqref{sequence1} by $\pi/2$. Accordingly, in the remainder of this paper, we only consider \eqref{sequence1} in 2D.
There is only one 3D version of the sequence, which reads:
 \begin{equation}\label{sequence3}
\begin{CD}
0@>>>\mathbb{R}@>>>H({\grad}, \Omega) @>{\grad}>> {H}(\curl, \Omega)@>\curl>> {H}(\mathrm{div}, \Omega)  @>\mathrm{div} >> L^{2}(\Omega)  @ > >>  0.
\end{CD}
\end{equation}
Here, the 3D curl has its usual meaning, mapping a vector valued function to a vector valued function:
\begin{align}\label{3Dcurl}
\curl\bm{u}:=(\partial_{2}u_{3}-\partial_{3}u_{2}, \partial_{3}u_{1}-\partial_{1}u_{3}, \partial_{1}u_{2}-\partial_{2}u_{1})^{T}.
\end{align}
When there is no possible confusion, we will use $\curl$ to denote both \eqref{2Dcurl} and \eqref{3Dcurl}, as determined by the spatial dimension.
The sequences \eqref{sequence1}, \eqref{sequence2} and \eqref{sequence3} are all exact on contractible domains. 
We will work with three known finite element subcomplexes of the de Rham complex.
In 2D these are:
\begin{align}
\label{2d-fe-seqs}
\begin{diagram}
0& \rTo &\mathbb{R} & \rTo &  \mathcal{Q}^{-}_{r}\Lambda^{0}(\mathcal{T}_{h}^{2}) & \rTo^{\curl} & \mathcal{Q}^{-}_{r}\Lambda^{1}(\mathcal{T}_{h}^{2}) & \rTo^{\div} &\mathcal{Q}^{-}_{r}\Lambda^{2}(\mathcal{T}_{h}^{2}) & \rTo^{} & 0, 
\\
0& \rTo &\mathbb{R} & \rTo &  \mathcal{S}_{r+2}\Lambda^{0}(\mathcal{T}_{h}^{2}) & \rTo^{\curl} & \mathcal{S}_{r+1}\Lambda^{1}(\mathcal{T}_{h}^{2}) & \rTo^{\div} &\mathcal{S}_{r}\Lambda^{2}(\mathcal{T}_{h}^{2}) & \rTo^{} & 0,\\
0& \rTo &\mathbb{R} & \rTo &  \mathcal{S}^{-}_{r}\Lambda^{0}(\mathcal{T}_{h}^{2}) & \rTo^{\curl} & \mathcal{S}^{-}_{r}\Lambda^{1}(\mathcal{T}_{h}^{2}) & \rTo^{\div} &\mathcal{S}^{-}_{r}\Lambda^{2}(\mathcal{T}_{h}^{2}) & \rTo^{} & 0,
\end{diagram}
\end{align}
namely,  the 2D Raviart-Thomas (tensor product) type finite element complex~\cite{arnold2015finite}, the 2D BDM type complex~\cite{arnold2014finite}, and the 2D trimmed serendipity complex~\cite{GK2016}, respectively, for $r\geq 1$.
The 3D families can be similarly denoted by
\begin{align}
\label{3d-fe-seqs}
\begin{diagram}
0& \rTo &\mathbb{R} & \rTo &  \mathcal{Q}^{-}_{r}\Lambda^{0}(\mathcal{T}_{h}^{3}) & \rTo^{\grad} & \mathcal{Q}^{-}_{r}\Lambda^{1}(\mathcal{T}_{h}^{3}) & \rTo^{\curl} &\mathcal{Q}^{-}_{r}\Lambda^{2}(\mathcal{T}_{h}^{3}) & \rTo^{\div} &\mathcal{Q}^{-}_{r}\Lambda^{3}(\mathcal{T}_{h}^{3})& \rTo^{} & 0, \\
0& \rTo &\mathbb{R} & \rTo &  \mathcal{S}_{r+3}\Lambda^{0}(\mathcal{T}_{h}^{3}) & \rTo^{\grad} & \mathcal{S}_{r+2}\Lambda^{1}(\mathcal{T}_{h}^{3}) & \rTo^{\curl} &\mathcal{S}_{r+1}\Lambda^{2}(\mathcal{T}_{h}^{3}) & \rTo^{\div} &\mathcal{S}_{r}\Lambda^{3}(\mathcal{T}_{h}^{3})& \rTo^{} & 0, \\
0& \rTo &\mathbb{R} & \rTo &  \mathcal{S}^{-}_{r}\Lambda^{0}(\mathcal{T}_{h}^{3}) & \rTo^{\grad} & \mathcal{S}^{-}_{r}\Lambda^{1}(\mathcal{T}_{h}^{3}) & \rTo^{\curl} &\mathcal{S}^{-}_{r}\Lambda^{2}(\mathcal{T}_{h}^{3}) & \rTo^{\div} &\mathcal{S}^{-}_{r}\Lambda^{3}(\mathcal{T}_{h}^{3})& \rTo^{} & 0. 
\end{diagram}
\end{align}

Some problems involve de Rham complexes with enhanced smoothness. For example, for Stokes flows, the velocity belongs to $ \left [{H}^{1}_{0}\right ]^{n}$ while the pressure is in $L^{2}_{0}$ with vanishing integration, and $\div:  \left [{H}^{1}_{0}\right ]^{n}\mapsto L^{2}_{0}(\Omega)$ is onto. In this case, we consider the 2D complex
 \begin{equation}\label{sequence1-smooth}
\begin{CD}
0@>>>\mathbb{R}@>>>H^{2}( \Omega) @>\curl>> \left [{H}^{1}( \Omega) \right ]^{2}@>\mathrm{div} >> L^{2}(\Omega)  @ > >>  0,
\end{CD}
\end{equation}
and the 3D version
 \begin{equation}\label{sequence3-smooth}
\begin{CD}
0@>>>\mathbb{R}@>>>H^{2}( \Omega) @>{\grad}>> {H}^{1}(\curl, \Omega)@>\curl>>  \left [{H}^{1}(\Omega) \right ]^{3} @>\mathrm{div} >> L^{2}(\Omega)  @ > >>  0,
\end{CD}
\end{equation}
where 
$$
H^{1}(\curl, \Omega):=\left\{u\in \left [H^{1}(\Omega)\right ]^{3}: \curl u\in \left [H^{1}(\Omega)\right ]^{3}\right\}.
$$
The subsequent discussions involve another 3D smooth de Rham complex
 \begin{equation}\label{sequence3-smooth2}
\begin{CD}
0@>>>\mathbb{R}@>>>H^{2}( \Omega) @>{\grad}>>  \left [{H}^{1}(\Omega)\right ]^{3}@>\curl>> {H}(\div, \Omega)  @>\mathrm{div} >> L^{2}(\Omega)  @ > >>  0.
\end{CD}
\end{equation}
The sequence \eqref{sequence3-smooth2} is exact on any contractible domain $\Omega$.

%%%%%%%%%%%%%%%%%%%%
\section{Serendipity and DoF-Transfer Operations}\label{sec:seren-transf}

In this section, we propose two operations for reducing piecewise-defined approximation spaces on cubical elements: a ``serendipity'' operation, which reduces local bubble functions, and a ``DoF-transfer'' operation, which enhances inter-element continuity by moving degrees of freedom. All the abovementioned properties are achieved, including preservation of the exactness.

%==============================

The serendipity operation, denoted by $\mathscr{Q}$, eliminates some interior DoFs and corresponding shape functions on individual elements.
Figure \ref{fig:serendipity} shows an example of the serendipity operation applied to the tensor product element $\mathcal{Q}_3^-\Lambda^0(\square_2)$, resulting in the serendipity element $\mathcal{S}_3\Lambda^0(\square_2)$.
On 0-forms, there is only one kind of serendipity operation, but on $k$-forms with $k\geq 1$, there are serendipity operations to both the regular and trimmed serendipity sequences; we denote both by $\mathscr{Q}$ as it will be clear from context which one is being used.

The DoF-transfer operation, denoted by $\mathscr{Q}$, ``moves'' two DoFs per edge from edge-association to vertex-association, interpolating either a directional derivative at the vertex (in the case of 0-forms) or a coordinate value of the vector field (in the case of 1-forms).
Figure~\ref{fig:transfer} shows an example of each type. 
The serendipity operation reduces local DoFs while the DoF-transfer operation reduces global DoFs and, as we will show, $\mathscr{T}$ and $\mathscr{Q}$ both preserve unisolvence and exactness.
\begin{center}
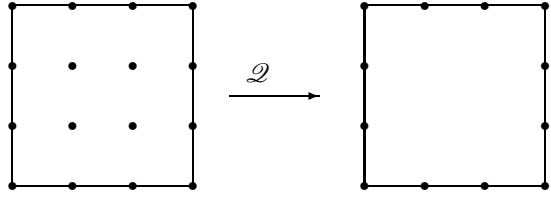
\begin{figure}
\setlength{\unitlength}{1.2cm}
\begin{picture}(2,2)(-4,0)
\put(0,0){
\begin{picture}(2,2)
\put(-1, 0){\line(1,0){2}} 
\put(1, 2){\line(-1,0){2}}
\put(-1,0){\line(0,1){2}}
\put(1, 2){\line(0, -1){2}}
\put(-1,0){\circle*{0.1}}
%\put(-1,0){\circle{0.2}}
\put(1.,0){\circle*{0.1}}
%\put(1.,0){\circle{0.2}}
\put(-1,2){\circle*{0.1}}
%\put(-1,2){\circle{0.2}}
\put(1,2){\circle*{0.1}}

\put(-1,0.666){\circle*{0.1}}
\put(-1,1.333){\circle*{0.1}}
\put(1,0.666){\circle*{0.1}}
\put(1,1.333){\circle*{0.1}}
\put(-0.333,0){\circle*{0.1}}
\put(0.333,0){\circle*{0.1}}
\put(-0.333,2){\circle*{0.1}}
\put(0.333,2){\circle*{0.1}}

\put(-0.333,0.666){\circle*{0.1}}
\put(0.333,0.666){\circle*{0.1}}
\put(-0.333,1.333){\circle*{0.1}}
\put(0.333,1.333){\circle*{0.1}}
\end{picture}
}

\put(1.5, 1){\vector(1, 0){1}}
\put(1.68, 1.15){$\mathscr{Q}$}

\put(4,0){
\put(-1, 0){\line(1,0){2}} 
\put(1, 2){\line(-1,0){2}}
\put(-1,0){\line(0,1){2}}
\put(1, 2){\line(0, -1){2}}
\put(-1,0){\circle*{0.1}}
%\put(-1,0){\circle{0.2}}
\put(1.,0){\circle*{0.1}}
%\put(1.,0){\circle{0.2}}
\put(-1,2){\circle*{0.1}}
%\put(-1,2){\circle{0.2}}
\put(1,2){\circle*{0.1}}

\put(-1,0.666){\circle*{0.1}}
\put(-1,1.333){\circle*{0.1}}
\put(1,0.666){\circle*{0.1}}
\put(1,1.333){\circle*{0.1}}
\put(-0.333,0){\circle*{0.1}}
\put(0.333,0){\circle*{0.1}}
\put(-0.333,2){\circle*{0.1}}
\put(0.333,2){\circle*{0.1}}
}
\end{picture}
\caption{An instance of the serendipity operation, taking $\mathcal{Q}_3^-\Lambda^0(\square_2)$ to $\mathcal{S}_3\Lambda^0(\square_2)$.}
\label{fig:serendipity}
\end{figure}
\end{center}

\begin{center}
\begin{figure}
\setlength{\unitlength}{0.9cm}
\begin{picture}(1,2)(-4,0)
\put(-1,0){ 
\put(-1, 0){\line(1,0){2}} 
\put(1, 2){\line(-1,0){2}}
\put(-1,0){\line(0,1){2}}
\put(1, 2){\line(0, -1){2}}
\put(-1,0){\circle*{0.1}}
%\put(-1,0){\circle{0.2}}
\put(1.,0){\circle*{0.1}}
%\put(1.,0){\circle{0.2}}
\put(-1,2){\circle*{0.1}}
%\put(-1,2){\circle{0.2}}
\put(1,2){\circle*{0.1}}
\put(-1,0.666){\circle*{0.1}}
\put(-1,1.333){\circle*{0.1}}
\put(1,0.666){\circle*{0.1}}
\put(1,1.333){\circle*{0.1}}
\put(-0.333,0){\circle*{0.1}}
\put(0.333,0){\circle*{0.1}}
\put(-0.333,2){\circle*{0.1}}
\put(0.333,2){\circle*{0.1}}
\put(-0.333,0.666){\circle*{0.1}}
\put(0.333,0.666){\circle*{0.1}}
\put(-0.333,1.333){\circle*{0.1}}
\put(0.333,1.333){\circle*{0.1}}
}
\put(0.5, 1){\vector(1, 0){1}}
\put(0.68, 1.15){$\mathscr{T}$}
\put(2.7,0){
\begin{picture}(2,2)
\put(-1, 0){\line(1,0){2}} 
\put(1, 2){\line(-1,0){2}}
\put(-1,0){\line(0,1){2}}
\put(1, 2){\line(0, -1){2}}
\put(-1,0){\circle*{0.1}}
\put(-1,0){\circle{0.2}}
\put(1.,0){\circle*{0.1}}
\put(1.,0){\circle{0.2}}
\put(-1,2){\circle*{0.1}}
\put(-1,2){\circle{0.2}}
\put(1,2){\circle*{0.1}}
\put(1,2){\circle{0.2}}

\put(-0.333,0.666){\circle*{0.1}}
\put(0.333,0.666){\circle*{0.1}}
\put(-0.333,1.333){\circle*{0.1}}
\put(0.333,1.333){\circle*{0.1}}
\end{picture}

}

\put(6.3,0){ 
\put(-1, 0){\line(1,0){2}} 
\put(1, 2){\line(-1,0){2}}
\put(-1,0){\line(0,1){2}}
\put(1, 2){\line(0, -1){2}}

%\put(-0,0){\circle*{0.1}}
%\put(0.5,0){\circle*{0.1}}
\put(-0.5,2){\vector(0, 1){0.5}}
\put(0.5,2){\vector(0, 1){0.5}}
\put(0,2){\vector(0, 1){0.5}}
\put(-0.5,0){\vector(0, -1){0.5}}
\put(0.5,0){\vector(0, -1){0.5}}
\put(0,0){\vector(0, -1){0.5}}
\put(1, 0.5){\vector( 1, 0){0.5}}
\put(1, 1.5){\vector( 1, 0){0.5}}
\put(1, 1){\vector(1, 0){0.5}}
\put(-1, 0.5){\vector( -1, 0){0.5}}
\put(-1, 1.5){\vector( -1, 0){0.5}}
\put(-1, 1){\vector(-1, 0){0.5}}
%\put(0.5,2){\circle*{0.1}}
%\put(0,0){\circle*{0.1}}
%\put(-0,2){\circle*{0.1}}
%\put(0.333,2){\circle*{0.1}}
\put(-0.4, 0.8){+12}
%\put(-0.333,0.666){\circle*{0.1}}
%\put(0.333,0.666){\circle*{0.1}}
%\put(-0.333,1.333){\circle*{0.1}}
%\put(0.333,1.333){\circle*{0.1}}
}
\put(8, 1){\vector(1, 0){1}}
\put(8.18, 1.15){$\mathscr{T}$}
\put(10.5,0){
\begin{picture}(2,2)
\put(-1, 0){\line(1,0){2}} 
\put(1, 2){\line(-1,0){2}}
\put(-1,0){\line(0,1){2}}
\put(1, 2){\line(0, -1){2}}
\put(-0.95,0){\circle*{0.1}}
\put(-1.05,0){\circle*{0.1}}
\put(0.95,0){\circle*{0.1}}
\put(1.05,0){\circle*{0.1}}
\put(-0.95,2){\circle*{0.1}}
\put(-1.05,2){\circle*{0.1}}
\put(0.95,2){\circle*{0.1}}
\put(1.05,2){\circle*{0.1}}

\put(0,0){\vector(0, -1){0.5}}
\put(0,2){\vector(0, 1){0.5}}
\put(1, 1){\vector(1, 0){0.5}}
\put(-1, 1){\vector(-1, 0){0.5}}

\put(-0.4, 0.8){+12}
\end{picture}
}
\end{picture}

\caption{An instance of the DoF-transfer operation, taking $\mathcal{Q}_3^-\Lambda^0(\square_2)$ to $\mathcal{G}_3^-\Lambda^0(\square_2)$ and taking $\mathcal{Q}_3^-\Lambda^1(\square_2)$ to $\mathcal{G}_3^-\Lambda^1(\square_2)$.}
\label{fig:transfer}
\end{figure}
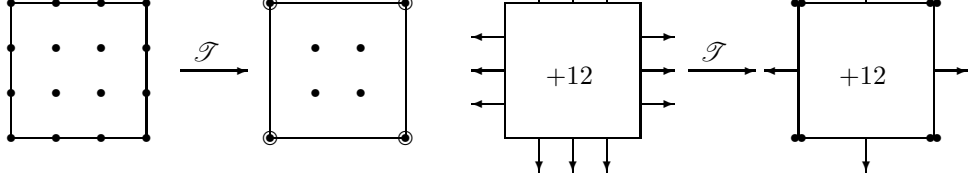
\end{center}

\noindent
Our subsequent constructions of finite element sequences are summarized in diagrams \eqref{3d-diagram-serendipity} and \eqref{3d-diagram-trimmedSrdp}. 
They employ the finite element sequences for $\mathcal{Q}^-\Lambda^{\bs}$, $\mathcal{S}\Lambda^{\bs}$, and $\mathcal{S}^-\Lambda^{\bs}$ given in \eqref{2d-fe-seqs}.
%The serendipity construction and DoF-transferring of the 3D complexes can be also summarized in such commuting diagrams. 

\begin{equation}\label{3d-diagram-serendipity}
\begin{tikzcd}
& \mathcal{Q}_{r+2}^-\Lambda^{0} \arrow[dl, "\mathscr{T}"] \arrow[rr, "\curl"] \arrow[dd, "\mathscr{Q}"]& & \mathcal{Q}_{r+2}^-\Lambda^{1} \arrow[dl, "\mathscr{T}"]\arrow[rr, "\div"]  \arrow[dd, "\mathscr{Q}"] & & \mathcal{Q}_{r+2}^-\Lambda^{2} \arrow[dl, "\mathscr{T}"]\arrow[dd, "\mathscr{Q}"] \\
 \mathcal{G}^{-}_{r+2}\Lambda^{0} \arrow[rr, crossing over, "\curl"] \arrow[dd, crossing over, "\mathscr{Q}"] & & \mathcal{G}^{-}_{r+2}\Lambda^{1}   \arrow[rr, crossing over, "\div"]  & &  \mathcal{G}^{-}_{r+2}\Lambda^{2} \\
& \mathcal{S}_{r+2}\Lambda^{0} \arrow[dl, "\mathscr{T}"] \arrow[rr] &  & \mathcal{S}_{r+1}\Lambda^{1} \arrow[dl, "\mathscr{T}"]  \arrow[rr]  & & \mathcal{S}_{r}\Lambda^{2} \arrow[dl, "\mathscr{T}"] \\
 \mathcal{A}_{r+2}\Lambda^{0} \arrow[rr, "\curl"] & &  \mathcal{A}_{r+1}\Lambda^{1}  \arrow[uu, leftarrow, crossing over, "\mathscr{Q}"]  \arrow[rr, "\div"]  & &  \mathcal{A}_{r}\Lambda^{2} \arrow[uu, leftarrow, crossing over, "\mathscr{Q}"]\\
\end{tikzcd}
\end{equation}

In diagram \eqref{3d-diagram-serendipity}, we start from the standard tensor product complex $\mathcal{Q}^-\Lambda^{\bs}$ (top, back), with the differential operators as shown.
%Here, the curl is defined by \eqref{2Dcurl}.
By a serendipity operation $\mathscr{Q}$, we obtain the BDM-type complex $\mathcal{S}\Lambda^{\bs}$  (bottom, back).
By the DoF-transfer operation $\mathscr{T}$, we get the Hermite complex $\mathcal{G}^{-}\Lambda^{\bs}$ (top, front). 
Applying the serendipity operation to $\mathcal{G}^{-}\Lambda^{\bs}$ or DoF-transferring $\mathcal{S}\Lambda^{\bs}$, we obtain the Adini complex $\mathcal{A}\Lambda^{\bs}$ (bottom, front).  
We will show that the diagram \eqref{3d-diagram-serendipity} commutes, i.e.\ that the operations $\mathscr{Q}$, $\mathscr{T}$ and the differential operators commute with each other. 
The extension of this diagram to complexes in 3D is straightforward.

A variant of this construction is described in diagram \eqref{3d-diagram-trimmedSrdp}.
Here, we start again with $\mathcal{Q}^-\Lambda^{\bs}$ but use a serendipity operation that results in the trimmed serendipity sequence  $\mathcal{S}^-\Lambda^{\bs}$.
Applying both $\mathscr{Q}$ and $\mathscr{T}$ gives a ``trimmed-Adini'' complex, denoted $\mathcal{A}^-\Lambda^{\bs}$.
As indicated in the diagram, there are some equalities between trimmed and non-trimmed-Adini elements when $k=0$ or $n$, based on identities for serendipity-type elements, but for $0<k<n$, the Adini and trimmed-Adini elements are truly distinct.

\begin{equation}\label{3d-diagram-trimmedSrdp}
\begin{tikzcd}
& \mathcal{Q}_{r+2}^-\Lambda^{0} \arrow[dl, "\mathscr{T}"] \arrow[rr, "\curl"] \arrow[dd, "\mathscr{Q}"]& & \mathcal{Q}_{r+2}^-\Lambda^{1} \arrow[dl, "\mathscr{T}"]\arrow[rr, "\div"]  \arrow[dd, "\mathscr{Q}"] & & \mathcal{Q}_{r+2}^-\Lambda^{2} \arrow[dl, "\mathscr{T}"]\arrow[dd, "\mathscr{Q}"] \\
 \mathcal{G}^{-}_{r+2}\Lambda^{0} \arrow[rr, crossing over, "\curl"] \arrow[dd, crossing over, "\mathscr{Q}"] & & \mathcal{G}^{-}_{r+2}\Lambda^{1}   \arrow[rr, crossing over, "\div"]  & &  \mathcal{G}^{-}_{r+2}\Lambda^{2} \\
& \mathcal{S}_{r+2}^-\Lambda^{0} \arrow[dl, "\mathscr{T}"] \arrow[rr] &  & \mathcal{S}_{r+2}^-\Lambda^{1} \arrow[dl, "\mathscr{T}"]  \arrow[rr]  & & \mathcal{S}_{r+2}^-\Lambda^{2} \arrow[dl, "\mathscr{T}"] \\
\mathcal{A}_{r+2}^{-}\Lambda^{0} \arrow[rr, "\curl"] & &  \mathcal{A}_{r+2}^{-}\Lambda^{1}  \arrow[uu, leftarrow, crossing over, "\mathscr{Q}"]  \arrow[rr, "\div"]  & &  \mathcal{A}_{r+2}^{-}\Lambda^{2} \arrow[uu, leftarrow, crossing over, "\mathscr{Q}"]\\[-6mm]
 = & & & & = \\[-6mm]
\mathcal{A}_{r+2}\Lambda^{0} & & & & \mathcal{A}_{r+1}\Lambda^{2} 
\end{tikzcd}
\end{equation}

We now formalize the operations $\mathscr{T}$ and $\mathscr{Q}$.
Degrees of freedom on an $n$-cube $\square_n$ when $P$ is a space of polynomial differential $k$-forms are typically described by
\begin{align}\label{dof-formal-feec}
u \longmapsto \int_f(\text{tr}_fu)\wedge q,\qquad q\in \mathcal{I}_f, 
\end{align}
where $f\prec\square_n$ is a sub-face of the cube and $\mathcal{I}_f$ is a space of polynomial differential $(n-k)$-forms.
We call $\mathcal{I}_f$ the index space associated to $f$ and there is a canonical association
\begin{align}\label{dof-index-assoc}
D\longleftrightarrow \bigoplus_{d=k}^n\bigoplus_{\footnotesize \begin{array}{c} f\prec\square_n \\ \dim f = d\end{array}}\mathcal{I}_f
\end{align}
For instance, DoFs associated to $\mathcal{Q}_r^-\Lambda^k(\square_n)$ have $\mathcal{I}_f:=\mathcal{Q}_{r-1}^-\Lambda^{d-k}(f)$ for each $f\prec\square_n$ \cite[equation (13)]{arnold2015finite}.
For finite elements conforming with respect to the standard deRham sequence, i.e.\ \eqref{sequence1} and \eqref{sequence3}, DoFs are associated to faces of dimension 0 only in the case of 0-forms ($k=0$).
In those cases,  $\dim\mathcal{I}_f=1$ and (\ref{dof-formal-feec}) is interpreted as evaluation of the 0-form at the vertex.

The non-standard elements considered here have two additional kinds of DoFs.
For a 0-form $u$ and vertex $x$, we allow partial derivative evaluation at vertices, i.e.\
\begin{align}\label{dof-part-deriv}
u\longmapsto \partial_{i}u(x)
\end{align}
where $i$ indicates the direction of an edge $e_i$ incident to $x$.
For a 1-form $\bm{v}$ and vertex $x$, we allow evaluation of the 1-form at the vertex, i.e.\
\begin{align}\label{dof-vec-eval}
\bm{v}\mapsto \bm{v}(x),
\end{align}
where $\bm{v}(x)$ is a vector in $\mathbb{R}^2$ or $\mathbb{R}^3$.
While these DoFs are associated to a vertex of the geometry in the global setting, we observe that they require both a vertex and an incident edge direction to be characterized  locally, e.g.\ the component of the vector $\bm{v}(x)$ in the 1-form case.
Therefore, we continue to treat these as edge DoFs to simplify the  upcoming formalism.

Let $n=2$ or 3.
The serendipity operation $\mathscr{Q}:(P_1,\square_n,D_1)\rightarrow (P_2,\square_n,D_2)$ can be thought of as acting component-wise on $P_1$ and $D_1$ and should satisfy the following conditions:
\begin{itemize}
\item $(P_1,\square_n,D_1)$ and $(P_2,\square_n,D_2)$ are finite elements, and $\mathscr{Q}$ is onto.
\item $P_2\subset P_1$ and $\mathscr{Q}:P_1\rightarrow P_2$ is a projection.
\item 
%$\mathscr{Q}:D_1\rightarrow D_2$ only affects DoFs associated to faces $f\prec\square_n$ with $\dim f\geq 2$.
$\mathscr{Q}$ is the identity on DoFs associated to vertices or edges of $\square_n$.
For $\max(k,2)\leq d\leq n$ and each $f\prec\square_n$ with $\dim f=d$, let $\mathcal{I}_{f,1}$ and $\mathcal{I}_{f,2}$ denote the index spaces from (\ref{dof-index-assoc}) from $D_1$ and $D_2$, respectively.
Then $\mathcal{I}_{f,2}\subseteq \mathcal{I}_{f,1}$ and $\mathscr{Q}:\mathcal{I}_{f,1}\rightarrow \mathcal{I}_{f,2}$ is a projection.
$\mathscr{Q}:D_1\rightarrow D_2$ is the linear morphism that is determined by these conditions via (\ref{dof-index-assoc}).
\end{itemize}

The DoF-transferring operation $\mathscr{T}:(P_1,\square_n,D_1)\rightarrow (P_2,\square_n,D_2)$ can be thought of as acting component-wise on $P_1$ and $D_1$, and should satisfy the following conditions:
\begin{itemize}
\item $(P_1,\square_n,D_1)$ and $(P_2,\square_n,D_2)$ are finite elements, and $\mathscr{Q}$ is onto.
\item $P_1=P_2$ and $\mathscr{T}:P_1\rightarrow P_2$ is the identity map.
\item $D_1$ only has DoFs of type (\ref{dof-formal-feec}).
\item $\mathscr{T}:D_1\rightarrow D_2$ only affects DoFs associated to edges $e\prec\square_n$ with $\dim\mathcal{I}_e\geq 2$, and leaves all other DoFs fixed.  
On each such edge $e$, $\mathscr{T}$ changes two DoFs from type (\ref{dof-formal-feec}) to type \eqref{dof-part-deriv} or \eqref{dof-vec-eval}, according to whether $P_1$ is a space of 0-forms or 1-forms, respectively.

\end{itemize}

\begin{lemma}\label{lem:commuting}
If $\mathscr{Q}$ leads to the same local shape functions for both the DoF-transferred and the original elements, then diagrams \eqref{3d-diagram-serendipity} and \eqref{3d-diagram-trimmedSrdp} are commutative.
\end{lemma}
\begin{proof}
The result holds because $\mathscr{T}\circ \mathscr{Q}$ and $\mathscr{Q}\circ \mathscr{T}$ yield elements with same local spaces and same inter-element continuity. 
\end{proof}
The condition for $\mathscr{Q}$ in Lemma \ref{lem:commuting} is fulfilled in all the examples below for the Hermite, Adini and trimmed-Adini families.
We now verify some properties of the two operations.

\begin{lemma}\label{lem:DoF-unisolvence}
The DoF-transfer operation preserves unisolvence, i.e.\ if $(P,\square_n,D)$ defines a unisolvent finite element, then $\mathscr{T}(P,\square_n,D)$ is also unisolvent.
\end{lemma}
\begin{proof}
%The only changes are the edge DoFs (including possible vertex DoFs) for the 0- and 1-forms. 
The operation $\mathscr{T}$ only changes the type of some edge DoFs, from type \eqref{dof-formal-feec} to type \eqref{dof-part-deriv} or \eqref{dof-vec-eval}, but preserves association of DoFs to edges, as discussed after \eqref{dof-vec-eval}.
The resulting set of DoFs is still linearly independent on each edge, so unisolvence of the element is preserved by the operation.
\end{proof}

\begin{lemma}\label{lem:DoF-exactness}
The DoF-transfer operation preserves exactness, i.e.\ if $(P^{\bs},\square_n,D^{\bs})$ defines an exact sequence of finite elements, then $\mathscr{T}(P^{\bs},\square_n,D^{\bs})$ is also exact.
\end{lemma}
\begin{proof}
Since the DoF-transfer only changes some 0- and 1-forms, the exactness at $\mathscr{T}(P^{k}, \square_n,D^{k})$, $k\geq 3$, holds due to the exactness of the original sequence $(P^{\bs},\square_n,D^{\bs})$. For the exactness at $k=1$, we observe that if $u\in \mathscr{T}(P^{1}, \square_n,D^{1})$ and $du=0$, then there exists $\phi\in (P^{0}, \square_n,D^{0})$ such that $\grad \phi=u$, by the exactness of  $(P^{\bs},\square_n,D^{\bs})$  at index 1 and the fact that $\mathscr{T}(P^{\bs},\square_n,D^{\bs})$ are subspaces of $(P^{\bs},\square_n,D^{\bs})$. Since $u$ is $C^{0}$ at vertices, we conclude that $\phi$ is correspondingly $C^{1}$ at vertices. This implies that $\phi\in \mathscr{T}(P^{0}, \square_n,D^{0})$ and shows the exactness of  $\mathscr{T}(P^{\bs},\square_n,D^{\bs})$ at index 1. It only remains to check the exactness at index 2. This follows by a dimension count.  In $n$ dimensions, on each element, $n$ DoFs are added per vertex and $n$ DoFs are removed per edge for $0$-forms and likewise for $1$-forms.
Thus, the alternating sum of dimension counts of spaces in the sequence is unaffected, meaning exactness is preserved.
Therefore the exactness of the whole sequence $\mathscr{T}(P^{\bs},\square_n,D^{\bs})$ holds.
\end{proof}

\begin{remark}
Starting with 2-forms, the DoF-transferred sequence branches into the standard finite element sequence. A direct check of the exactness at the space of two forms is not trivial. For example, the 2D 1-forms ($H(\div)$ element) with vertex continuity and the 2-forms of piecewise polynomials have been analyzed in Stenberg \cite{Stenberg2010} by a macroelement technique. From the point of view presented in this paper and \cite{christiansen2016nodal}, the exactness seems more clear and natural.
\end{remark}

In subsequent discussions, we use the term ``bubble function'' to mean a shape function on $\square_n$ for which all DoFs vanish on all $\ell$-dimensional facets of $\square_n$ for $0\leq\ell\leq n-1$.

\begin{lemma}\label{lem:serendipity-bubbles}
The DoF-transfer operation does not change the local space of bubble functions on facets of dimension $\ell\geq 2$.
\end{lemma}
\begin{proof}
Since the DoF-transferred sequence $\mathscr{T}(P^{\bs},\square_n,D^{\bs})$ has stronger inter-element continuity than $(P^{\bs},\square_n,D^{\bs})$, we conclude that the bubble spaces of  $\mathscr{T}(P^{\bs},\square_n,D^{\bs})$ are contained in those of $(P^{\bs},\square_n,D^{\bs})$.  
On the other hand, the DoF-transfer operation does not change the interior DoFs associated to facets with dimension $\ell\geq 2$.  
So, the spaces of bubble functions of $\mathscr{T}(P^{\bs},\square_n,D^{\bs})$ and $(P^{\bs},\square_n,D^{\bs})$ have the same dimension.
Therefore, the spaces of bubble functions for these two families are the same. 
\end{proof}

We now examine properties of the serendipity operation.
First, we address the issue of unisolvence.
Degrees of freedom of type \eqref{dof-formal-feec} are unisolvent if and only if on each $k$ dimensional face $f\in \square_{m}, k\leq  m\leq n$, the trace space with vanishing boundary conditions on $f$ (the space of bubble functions on $f$) coincides with  the local space $\mathcal{I}_{f}$ used in the definition of degrees of freedom in \eqref{dof-formal-feec}, i.e., $\ast\tr_{f} (P^{k}\cap H_{0}\Lambda^{k})=\mathcal{I}_{f}$. Here, the Hodge star is used since in the language of differential forms, the inner product is defined by $(u, v)_{f}:=\displaystyle\int_{f}\ast u\wedge v\, dx$. In terms of vector proxies, the test space $\mathcal{I}_{f}$ is just the same as the space of bubbles on $f$. We assume that the serendipity operation only affects degrees of freedom of type \eqref{dof-formal-feec}.  If the image of a the serendipity operation $\mathscr{Q}:(P_1,\square_n,D_1)\rightarrow (P_2,\square_n,D_2)$ satisfies the corresponding relation $\ast\tr_{f}\left( P_{2}^{\bs}\cap H_{0}\Lambda^{\bs}\right)= \mathcal{I}_{f, 2}$, then unisolvence will be preserved.

  For exactness, consider the diagram \eqref{diagram:long}:
\begin{align}\label{diagram:long}
\begin{diagram}
&  & & &0  & & 0&  &0 &  & 0  & & \\
&  & &  & \dTo^{}  & &\dTo^{}  &  &\dTo&  & \dTo  &  & \\
&  &0 & \rTo & \ker(\mathscr{Q}_0)  & \rTo^{d} & \ker(\mathscr{Q}_1)  & \rTo^{d} &\cdots  & \rTo^{d} & \ker(\mathscr{Q}_n)  & \rTo & 0\\
&  & &  & \dTo^{\iota}  & &\dTo^{\iota}  &  &\dTo^{\iota}&  & \dTo^{\iota}  &  & \\
 0&\rTo  &\mathbb{R} & \rTo & V_{0}  & \rTo^{d} &V_{1} & \rTo^{d} & \cdots & \rTo^{d} & V_{n}  & \rTo & 0\\
&  &  &  & \dTo^{\mathscr{Q}_0}  & &\dTo^{\mathscr{Q}_1}  &  &\dTo^{\mathscr{Q}_2}&  & \dTo^{\mathscr{Q}_3}  &  & \\
 0&\rTo  &\mathbb{R} & \rTo & \tilde{V}_{0}  & \rTo^{d} &\tilde{V}_{1} & \rTo^{d} & \cdots & \rTo^{d} & \tilde{V}_{n}  & \rTo & 0\\
&  &  &  & \dTo  & &\dTo  &  &\dTo&  & \dTo  &  & \\
  &  &   & &0  & & 0&  &0 &  & 0  & & 
\end{diagram}
\end{align}

The first row of \eqref{diagram:long} consists of the kernel spaces of the serendipity operation, indicating the subspaces that have been eliminated. 
The second row is the finite element sequence that we start with and the last row is the reduced sequence.
By definition, each column 
\begin{align}
\begin{diagram}
0 & \rTo & \ker(\mathscr{Q}_k)  & \rTo^{\iota} & V_k & \rTo^{\mathscr{Q}_k} & \tilde V_k & \rTo^{} & 0,
\end{diagram}
\end{align}
is a short exact sequence, where $\iota$ is the inclusion map. 
 Assume that the row in the middle, i.e.
\begin{align}\label{seq:original}
\begin{diagram}
 0&\rTo  &\mathbb{R} & \rTo & V_{0}  & \rTo^{d} &V_{1} & \rTo^{d} & \cdots & \rTo^{d} & V_{n}  & \rTo & 0,
\end{diagram}
\end{align}
is exact. From a general algebraic result, the exactness of either the first or last row of \eqref{diagram:long} implies exactness of the other (c.f.\ \cite{dummit2004abstract} for the general result and \cite[Proposition 5.16, Proposition 5.17]{christiansen2011topics} for an application in the context of finite element systems).
In particular, the exactness of the kernel sequence
\begin{align}\label{seq:kernel}
\begin{diagram}
0 & \rTo & \ker(\mathscr{Q}_0)  & \rTo^{d} & \ker(\mathscr{Q}_1)  & \rTo^{d} &\cdots  & \rTo^{d} & \ker(\mathscr{Q}_n)  & \rTo & 0,
\end{diagram}
\end{align}
 implies that of the reduced sequence 
\begin{align}\label{seq:reduced}
\begin{diagram}
 0&\rTo  &\mathbb{R} & \rTo & \tilde{V}_{0}  & \rTo^{d} &\tilde{V}_{1} & \rTo^{d} & \cdots & \rTo^{d} & \tilde{V}_{n}  & \rTo & 0,
\end{diagram}
\end{align}
and vice versa. 
The following lemma summarizes this result.
\begin{lemma}\label{lem:serendipity-exact}
In any setting of the form \eqref{diagram:long}, if any two rows of the diagram are exact then the third is exact as well.
In particular, exactness of the first and second rows implies that the serendipity operation preserves exactness.
\end{lemma}

\begin{remark}
From the perspective of finite element systems, the serendipity operation reduces local (bubble) spaces on various dimensions and the DoF-transfer only changes the spaces carried by zero dimensional cells (vertex spaces) and the trace/restriction from one dimensional cells to their boundaries.  Although a rigorous theory has not been studied in the literature, we observe that the structures of the FES are not changed for two and higher dimensions. The DoF-transfer defined in this section changes the  system associated with one dimensional cells from Lagrange to Hermite type. This partly explains the fact that face and interior bubble functions are not affected by the DoF-transfer. 
\end{remark}

%%%%%%%%%%%%%%%%%%%%%%%%%%%%%%%%%%%%%%%%%%%%%%%%%%%%%%%%%%%%%%%%%%

\section{Hermite, Adini and trimmed-Adini complexes}\label{sec:elements}

In this section we introduce the Hermite, Adini and trimmed-Adini complexes obtained from standard complexes by the serendipity and DoF-transfer operations. We describe the function spaces and degrees of freedom definitions in two and three space dimensions. 

\subsection{Hermite complex}
\subsubsection{Two dimensions}

The 2D Hermite sequence is denoted by
\begin{align}\label{2d-hermite-sequence}
\begin{diagram}
0& \rTo & \mathbb{R} & \rTo & \mathcal{G}^{-}_{r}\Lambda^{0}(\mathcal{T}_{h}) & \rTo^{\curl} &\mathcal{G}^{-}_{r}\Lambda^{1}(\mathcal{T}_{h}) & \rTo^{\div} &\mathcal{G}^{-}_{r}\Lambda^{2}(\mathcal{T}_{h})& \rTo^{} & 0.
\end{diagram}
\end{align}
The shape function spaces coincide with those of the tensor product family on squares, i.e.\ $\mathcal{Q}^-\Lambda^k$, or Raviart-Thomas elements.
However, the Hermite family is only defined when $r\geq 3$, as smaller $r$ values do not provide enough degrees of freedom to ensure the requisite continuity at vertices for $k=0$ and $k=1$.  The lowest order case is shown in Figure \ref{fig:hermite2d}.

\begin{center}
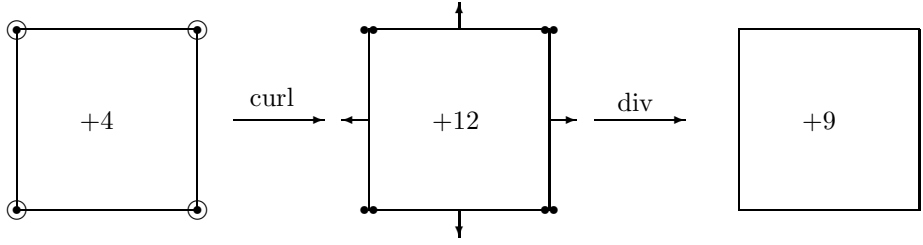
\begin{figure}
\setlength{\unitlength}{1.2cm}
\begin{picture}(2,2)(-2,0)
\put(0,0){
\begin{picture}(2,2)
\put(-1, 0){\line(1,0){2}} 
\put(1, 2){\line(-1,0){2}}
\put(-1,0){\line(0,1){2}}
\put(1, 2){\line(0, -1){2}}
\put(-1,0){\circle*{0.1}}
\put(-1,0){\circle{0.2}}
\put(1.,0){\circle*{0.1}}
\put(1.,0){\circle{0.2}}
\put(-1,2){\circle*{0.1}}
\put(-1,2){\circle{0.2}}
\put(1,2){\circle*{0.1}}
\put(1,2){\circle{0.2}}
%\put(-0.2, 0.8){\circle*{0.1}}
%\put(0.2, 0.8){\circle*{0.1}}
%\put(-0.2, 1.2){\circle*{0.1}}
%\put(0.2, 1.2){\circle*{0.1}}
\put(-0.3, 0.9){+4}
\end{picture}
}

\put(1.5, 1){\vector(1, 0){1}}
\put(1.68, 1.15){$\curl$}

\put(4,0){
\put(-1, 0){\line(1,0){2}} 
\put(1, 2){\line(-1,0){2}}
\put(-1,0){\line(0,1){2}}
\put(1, 2){\line(0, -1){2}}
\put(-0.95,0){\circle*{0.1}}
\put(-1.05,0){\circle*{0.1}}
\put(1.05,0){\circle*{0.1}}
\put(0.95,0){\circle*{0.1}}
\put(-0.95,2){\circle*{0.1}}
\put(-1.05,2){\circle*{0.1}}
\put(1.05,2){\circle*{0.1}}
\put(0.95,2){\circle*{0.1}}
\put(0, 2){\vector(0,1){0.3}}
\put(0, 0){\vector(0,-1){0.3}}
\put(-1, 1){\vector(-1,0){0.3}}
\put(1, 1){\vector(1,0){0.3}}
%\put(1, 1){\circle*{0.1}}
\put(-0.3, 0.9){+12}
%\put(0.2, 1){\circle*{0.1}}
}

\put(5.5, 1){\vector(1, 0){1}}
\put(5.75, 1.1){{$\div$}}

\put(8,0){
\begin{picture}(2,2)
\put(-1, 0){\line(1,0){2}} 
\put(1, 2){\line(-1,0){2}}
\put(-1,0){\line(0,1){2}}
\put(1, 2){\line(0, -1){2}}
\put(-0.3, 0.9){+9}
\end{picture}
}

%\put(9, 1){\vector(1, 0){1}}
\end{picture}
\caption{The lowest order Hermite complex, with regularity $H(\curl)\rightarrow {H}(\div)\rightarrow L^{2}$. The local shape function spaces are $\mathcal{Q}_{3}\rightarrow \mathcal{P}_{3, 2}\times \mathcal{P}_{2, 3}\rightarrow \mathcal{Q}_{2}$, or, equivalently, $\mathcal{Q}_{3}\Lambda^0(\square_2)\rightarrow \mathcal{Q}_{3}\Lambda^1(\square_2)\rightarrow \mathcal{Q}_{3}\Lambda^2(\square_2)$. }
\label{fig:hermite2d}
\end{figure}
\end{center}

\paragraph{Space $\mathcal{G}^{-}_{r}\Lambda^{0}(\square_2)$.}
The shape function space of $\mathcal{G}^{-}_{r}\Lambda^{0}(\square_2)$ is $\mathcal{Q}_{r}$. The DoFs can be given by
\begin{itemize}
\item
function evaluation and first order derivatives at each vertex $x\in \mathcal{V}$:
$$
u(x), \quad \partial_{i}u(x), ~i=1, 2, 
$$
\item
moments on each edge $e\in \mathcal{E}$:
$$
\int_{e} u q\ dx, \quad q\in \mathcal{P}_{r-4}(e),
$$
\item
interior DoFs in each $K\in \mathcal{F}$:
$$
\int_{K}up\ dx, \quad p\in \mathcal{Q}_{r-2}(K).
$$
\end{itemize}
%The above degrees of freedom are unisolvent.

\paragraph{Space $\mathcal{G}^{-}_{r}\Lambda^{1}(\square_2)$.} The shape function space of $\mathcal{G}^{-}_{r}\Lambda^{1}(\square_2)$ is $\mathcal{P}_{r, r-1}\times \mathcal{P}_{r-1, r}$, the same as $\mathcal{Q}^-_r\Lambda^1(\square_2)$.  The DoFs can be given by
\begin{itemize}
\item
function evaluation at each vertex $x\in \mathcal{V}$:
$$
\bm{v}(x), 
$$
\item
moments of the normal components on each edge $e\in \mathcal{E}$:
$$
\int_{e}\bm{v}\cdot\bm{\nu}_{e}q\ dx, \quad q\in \mathcal{P}_{r-3}(e),
$$
\item
interior degrees of freedom in each $K\in\mathcal{F}$:
$$
\int_{K}\bm{v}\cdot\bm{\psi}\ dx, \quad \bm{\psi}\in \mathcal{P}_{r-1, r-2}(K)\times \mathcal{P}_{r-2, r-1}(K).
$$
\end{itemize}

\paragraph{Space $\mathcal{G}^{-}_{r}\Lambda^{2}(\square_2)$.} The space $\mathcal{G}^{-}_{r}\Lambda^{2}(\square_2)$ consists of  piecewise $\mathcal{Q}_{r-1}$ polynomials on each element.

\subsubsection{Three dimensions}

The 3D Hermite sequence is denoted by
\begin{align}\label{3dcomplex-hermite}
\begin{diagram}
0& \rTo &\mathbb{R} & \rTo &  \mathcal{G}^{-}_{r}\Lambda^{0}(\mathcal{T}_{h}) & \rTo^{\grad} & \mathcal{G}^{-}_{r}\Lambda^{1}(\mathcal{T}_{h}) & \rTo^{\curl} &\mathcal{G}^{-}_{r}\Lambda^{2}(\mathcal{T}_{h}) & \rTo^{\div} &\mathcal{G}^{-}_{r}\Lambda^{3}(\mathcal{T}_{h})& \rTo^{} & 0.
\end{diagram}
\end{align}

\paragraph{Space $\mathcal{G}^{-}_{r}\Lambda^{0}(\square_3)$.}
The shape function space of $\mathcal{G}^{-}_{r}\Lambda^{0}(\square_3)$ is  $\mathcal{Q}_{r}$. The DoFs can be given by
\begin{itemize}
\item
function evaluation and first order derivatives at each vertex $x\in \mathcal{V}$:
$$
u(x), \quad \partial_{i}u(x), ~i=1, 2,
$$
\item
moments on each edge $\forall e\in \mathcal{E}$:
$$
\int_{e} u q\ dx, \quad q\in \mathcal{P}_{r-4}(e), 
$$
\item
moments on each face $f\in \mathcal{F}$:
$$
\int_{f} u p\ dx, \quad p\in \mathcal{Q}_{r-2}(f), 
$$
\item
interior DoFs in each $K\in \mathcal{K}$:
$$
\int_{K}uw\ dx, \quad w\in \mathcal{Q}_{r-2}(K).
$$
\end{itemize}
%By  similar argument as the Lagrange type $\mathcal{Q}_{r}$ element, we can see that the above degrees of freedom are unisolvent.

\paragraph{Space $\mathcal{G}^{-}_{r}\Lambda^{1}(\square_3)$.} The discrete $H(\curl)$ space of the 3D Hermite complex has the shape function space  $\mathcal{P}_{r, r+1, r+1}\times \mathcal{P}_{r+1, r, r+1}\times \mathcal{P}_{r+1, r+1, r}$.  DoFs for a function $\bm{v}\in \mathcal{G}^{-}_{r}\Lambda^{1}(\square_3)$ can be given by 
\begin{itemize}
\item
function evaluation at each vertex $x\in \mathcal{V}$:
$$
\bm{v}(x),
$$
\item
moments of the tangential components on each edge $e\in \mathcal{E}$:
$$
\int_{e}\bm{v}\cdot\bm{\tau}_{e}q\ dx, \quad q\in \mathcal{P}_{r-3}(e),
$$
\item
moments of the tangential components on each face $f\in \mathcal{F}$:
$$
\int_{f}\left (\bm{v}\times \bm{\nu}_{f}\right )\cdot\bm{p}\ dx, \quad \bm{p}\in \mathcal{P}_{r-1, r-2}(f)\times \mathcal{P}_{r-2, r-1}(f),
$$
where $\bm{v}\times \bm{\nu}_{f}$ is treated as a 2-vector in the plane of the face,
\item
interior DoFs in each $K\in \mathcal{K}$:
$$
\int_{K}\bm{v}\cdot\bm{s}\ dx, \quad \bm{s}\in \mathcal{P}_{r-1, r-2, r-2}(K)\times \mathcal{P}_{r-2, r-1, r-2}(K)\times \mathcal{P}_{r-2, r-2, r-1}(K).
$$
\end{itemize}

\paragraph{Spaces $\mathcal{G}^{-}_{r}\Lambda^{2}(\square_3)$ and $\mathcal{G}^{-}_{r}\Lambda^{3}(\square_3)$.} The last two spaces in the sequence coincide with the standard Raviart-Thomas element of degree $r$ and the piecewise polynomial tensor product element $\mathcal{Q}_{r-1}$, i.e.\ $\mathcal{Q}^-_r\Lambda^2(\square_3)$ and $\mathcal{Q}^-_r\Lambda^3(\square_3)$.
The DoFs of $\mathcal{G}^{-}_{r}\Lambda^{2}(\square_3)$,  can be given by (c.f. \cite{Boffi.D;Brezzi.F;Fortin.M.2013a})
\begin{itemize}
\item
moments on each face $f\in \mathcal{F}$:
$$
\int_{f}\bm{u}\cdot\bm{\nu}_{f}q\ dx, \quad q\in \mathcal{Q}_{r-1}(f),
$$
\item
interior DoFs for each $K\in \mathcal{K}$:
$$
\int_{K}\bm{u}\cdot \bm{s}\ dx, \quad \bm{s}\in \mathcal{P}_{r-2, r-1 ,r-1}(K)\times \mathcal{P}_{r-1, r-2, r-1}(K)\times \mathcal{P}_{r-1,  r-1, r-2}(K).
$$
\end{itemize}

%%%%%%%%%%%%%%%%%%
\subsection{Adini complexes}

The Adini element for the plate problem~\cite{adini} can be regarded as a serendipity version of the cubical Hermite element. 
The shape function space of the 2D Adini element in the lowest order reads
$$
\mathcal{Q}_{1}+\mathrm{span}\left \{ x_{i}^{2}q~:~ 1\leq i\leq 2, ~q\in \mathcal{Q}_{1} \right \},
$$
which coincides with the shape function space of the serendipity element $\mathcal{S}_3\Lambda^0(\square_2)$. 
The degrees of freedom in this case can be given as the function evaluation and the first order derivatives at each vertex $x\in \mathcal{V}$.
%$$
%u(x), \quad \partial_{i}u(x), ~i=1, 2, \cdots, d.
%$$
We now explain how this definition can be expanded to an entire exact seqeuence of finite elements.

\subsubsection{Two dimensions}
The 2D Adini sequence is denoted by
\begin{align}\label{2d-adini-complex}
\begin{diagram}
0&\rTo&\mathbb{R} & \rTo & \mathcal{A}_{r+2}\Lambda^{0}(\mathcal{T}_{h}) & \rTo^{\curl} &\mathcal{A}_{r+1}\Lambda^{1}(\mathcal{T}_{h}) & \rTo^{\div} &\mathcal{A}_{r}\Lambda^{2}(\mathcal{T}_{h})& \rTo^{} & 0,  
\end{diagram}
\end{align}

\begin{center}
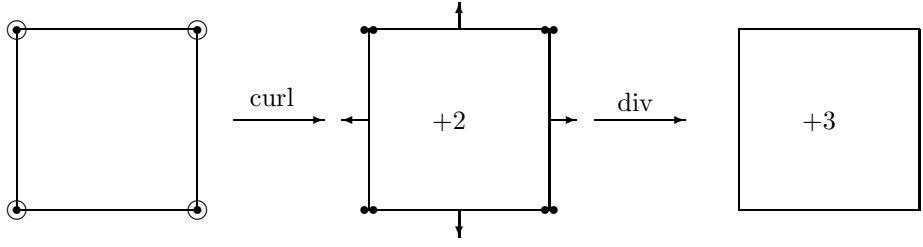
\begin{figure}

\setlength{\unitlength}{1.2cm}
\begin{picture}(2,2)(-2,0)
\put(0,0){
\begin{picture}(2,2)
\put(-1, 0){\line(1,0){2}} 
\put(1, 2){\line(-1,0){2}}
\put(-1,0){\line(0,1){2}}
\put(1, 2){\line(0, -1){2}}
\put(-1,0){\circle*{0.1}}
\put(-1,0){\circle{0.2}}
\put(1.,0){\circle*{0.1}}
\put(1.,0){\circle{0.2}}
\put(-1,2){\circle*{0.1}}
\put(-1,2){\circle{0.2}}
\put(1,2){\circle*{0.1}}
\put(1,2){\circle{0.2}}
\end{picture}
}

\put(1.5, 1){\vector(1, 0){1}}
\put(1.68, 1.15){$\curl$}

\put(4,0){
\put(-1, 0){\line(1,0){2}} 
\put(1, 2){\line(-1,0){2}}
\put(-1,0){\line(0,1){2}}
\put(1, 2){\line(0, -1){2}}
\put(-0.95,0){\circle*{0.1}}
\put(-1.05,0){\circle*{0.1}}
\put(1.05,0){\circle*{0.1}}
\put(0.95,0){\circle*{0.1}}
\put(-0.95,2){\circle*{0.1}}
\put(-1.05,2){\circle*{0.1}}
\put(1.05,2){\circle*{0.1}}
\put(0.95,2){\circle*{0.1}}
\put(0, 2){\vector(0,1){0.3}}
\put(0, 0){\vector(0,-1){0.3}}
\put(-1, 1){\vector(-1,0){0.3}}
\put(1, 1){\vector(1,0){0.3}}
%\put(1, 1){\circle*{0.1}}
%\put(-0.2, 1){\circle*{0.1}}
%\put(0.2, 1){\circle*{0.1}}
\put(-0.3, 0.9){+2}
}

\put(5.5, 1){\vector(1, 0){1}}
\put(5.75, 1.1){{$\div$}}

\put(8,0){
\begin{picture}(2,2)
\put(-1, 0){\line(1,0){2}} 
\put(1, 2){\line(-1,0){2}}
\put(-1,0){\line(0,1){2}}
\put(1, 2){\line(0, -1){2}}
%\put(0, 1){\circle*{0.1}}
%\put(-0.1, 0.8){\circle*{0.1}}
%\put(0.1, 0.8){\circle*{0.1}}
\put(-0.3, 0.9){+3}
\end{picture}
}

%\put(9, 1){\vector(1, 0){1}}
\end{picture}
\caption{The lowest order Adini complex, with regularity $H(\curl)\rightarrow {H}(\div)\rightarrow L^{2}$. The local shape function spaces are $\mathcal{S}_{3}\Lambda^{0}(\square_2)\rightarrow \mathcal{S}_{2}\Lambda^{1}(\square_2) \rightarrow  \mathcal{S}_{1}\Lambda^{2}(\square_2)$.}
\label{fig:adini2d}
\end{figure}
\end{center}

The space $\mathcal{A}_{r+2}\Lambda^{0}(\square_2)$ represents the Adini element of degree $r+2$ and $\mathcal{A}_{r}\Lambda^{2}(\square_2)=\mathcal{S}_r\Lambda^2(\square_2)=\mathcal{P}_{r}$, the space of piecewise polynomials of total degree at most $r$.
The lowest order case of the sequence is shown in Figure \ref{fig:adini2d}.

\paragraph{Space $\mathcal{A}_{r}\Lambda^{0}(\square_2)$.}
The shape function space is $\mathcal{S}_{r}\Lambda^{0}(\square_2)=\mathcal{P}_{r}+\mathrm{span}\left\{x^ry, xy^r\right\}$.
The DoFs are given by 
\begin{itemize}
\item
function evaluation and first order derivatives at each vertex $x\in \mathcal{V}$:
$$
u(x), \quad \partial_{i}u(x), ~i=1, 2, 
$$
\item
moments on each edge $e\in \mathcal{E}$:
$$
\int_{e} u q\ dx, \quad q\in \mathcal{P}_{r-4}(e),
$$
\item
interior DoFs in each $K\in \mathcal{F}$:
$$
\int_{K}up\ dx, \quad p\in \mathcal{P}_{r-4}(K).
$$
\end{itemize}

\paragraph{Space $\mathcal{A}_{r}\Lambda^{1}(\square_2)$.}
 The shape function space is
\begin{align}\label{space-BDM2D}
\mathcal{S}_{r}\Lambda^{1}(\square_2)=\mathcal{P}_{r}+\mathrm{span}\left\{\curl \left(x^{r+1}y\right ), \curl\left (xy^{r+1}\right )\right\},
\end{align}
which coincides with the BDM element of degree $r$. The degrees of freedom can be given by
\begin{itemize}
\item
function evaluation at each vertex $x\in \mathcal{V}$:
$$
\bm{u}_{i}(x),\quad i=1, 2,
$$
\item
moments on each edge $e\in \mathcal{E}$:
$$
\int_{e}\bm{u}\cdot\bm{\nu}_{e}q\ dx, \quad q\in \mathcal{P}_{r-2}(e),
$$
\item
interior DoFs in each $K\in \mathcal{K}$:
$$
\int_{K}\bm{u}\cdot\bm{p}\ dx, \quad \bm{p}\in \left [\mathcal{P}_{r-2}(K)\right ] ^2.
$$
\end{itemize}

\subsubsection{Three dimensions}

The 3D Adini sequence is denoted by
\begin{align}\label{3D-adini-complex}
\begin{diagram}
0& \rTo &\mathbb{R} & \rTo &  \mathcal{A}_{r+3}\Lambda^{0}(\mathcal{T}_{h}) & \rTo^{\grad} & \mathcal{A}_{r+2}\Lambda^{1}(\mathcal{T}_{h}) & \rTo^{\curl} &\mathcal{A}_{r+1}\Lambda^{2}(\mathcal{T}_{h}) & \rTo^{\div} &\mathcal{A}_{r}\Lambda^{3}(\mathcal{T}_{h})& \rTo^{} & 0.
\end{diagram}
\end{align}
%We explain the spaces below and show the global exactness of \eqref{3D-adini-complex}.

\paragraph{Space $\mathcal{A}_{r}\Lambda^{0}(\square_3)$.} The shape function space is the same as that of the serendipity element $\mathcal{S}_{r}\Lambda^{0}$, i.e. polynomials with superlinear degree at most $r$. For $r=3$,  the shape function space can be  equivalently represented as 
$$Q_{1}+\mathrm{span}\left \{x_{i}^{2}q~:~ 1\leq i\leq 3, ~q\in \mathcal{Q}_{1}\right \},$$ 
and $\mathcal{A}_{r}\Lambda^{0}(\mathcal{T}_{h})$ coincides with the 3D Adini element.   The degrees of freedom can be given by
\begin{itemize}
\item
function evaluation and first order derivatives at each vertex $x\in \mathcal{V}$:
$$
u(x), \quad \partial_{i}u(x), ~~i=1, 2, 3,
$$
\item
moments on each edge $e\in \mathcal{E}$:
$$
\int_{e}uq\ dx, \quad q\in \mathcal{P}_{r-4}(e), 
$$
\item
moments on each face $f\in \mathcal{F}$:
$$
\int_{f}uv\ dx, \quad v\in \mathcal{P}_{r-4}(f),
$$
\item 
interior DoFs in each $K\in \mathcal{K}$:
$$
\int_{K}uw\ dx, \quad w\in \mathcal{P}_{r-6}(K).
$$
\end{itemize}

\paragraph{Space $\mathcal{A}_{r}\Lambda^{1}(\square_3)$.}

The shape function space coincides with that of the $H(\curl)$ finite element given in \cite{arnold2014finite}, i.e.
$$
\mathcal{S}_{r}\Lambda^{1}(K):=\left [\mathcal{P}_{r}(K)\right ]^{3}+\mathrm{span}\left (
\left ( 
\begin{array}{c}
yz\left (w_{2}(x, z)-w_{3}(x, y)\right )\\
zx\left ( w_{3}(x, y)-w_{1}(y, z)\right )\\
xy\left ( w_{1}(y, z)-w_{2}( x, z)\right )
\end{array}
\right )
+\grad q(x, y, z)
 \right ),
$$
where $w_{i}\in \mathcal{P}_{r}(K)$ and $q$ is a polynomial on $K$ with superlinear degree at most $k+1$.
The DoFs for a vector function $\bm{u}$ can be given by
\begin{itemize}
\item
function evaluation at each vertex $x\in \mathcal{V}$: % 24 DoFs / element
$$
\bm{u}(x),
$$
\item
moments on each edge $e\in \mathcal{E}$: % 12(r-1) DoFs / element
$$
\int_{e}\bm{u}\cdot\bm{\tau}_{e}q\ dx, \quad q\in \mathcal{P}_{r-2}(e),
$$
\item
moments on each face $f\in \mathcal{F}$:
$$
\int_{f}\left (\bm{u}\times \bm{\nu}_{f}\right )\cdot \bm{p}\ dx, \quad \bm{p}\in \left[\mathcal{P}_{r-2}(f)\right]^2,
$$
where $\bm{u}\times \bm{\nu}_{f}$ is considered as a 2D vector,
\item
interior DoFs in $K\in \mathcal{K}$:
$$
\int_{K}\bm{u}\cdot\bm{s}\ dx, \quad \bm{s}\in \left [\mathcal{P}_{r-4}(K)\right ]^{3}.
$$
\end{itemize}

The last two spaces in the sequence are the same as the BDM type $H(\div)$ element and the piecewise polynomial space in \cite{arnold2014finite}. For completeness, we include the definitions of these spaces.
\paragraph{Space $\mathcal{A}_{r}\Lambda^{2}(\square_3)$.} The shape function space of the cubical BDM space is given by
$$
\mathrm{BDM}_{r}:=\left [\mathcal{P}_{r}(K)\right ]^{3}+\mathrm{span}\left \{  
\left (
\begin{array}{c}
yz\left (w_{2}(x, z)-w_{3}(x, y)\right )\\
zx\left (w_{3}(x, y)-w_{1}(y, z)\right )\\
yz\left (w_{1}(y, z)-w_{2}(x, z)\right )
\end{array}
\right )
\right\}.
$$
The DoFs for a vector function $\bm{v}$ can be taken as
\begin{itemize}
\item
moments on each face $f\in \mathcal{F}$:
$$
\int_{f}\bm{v}\cdot\bm{\nu}_{f}q\ dx,\quad q\in \mathcal{P}_{r}(f),
$$
\item
interior DoFs in $K\in \mathcal{K}$:
$$
\int_{K}\bm{v}\cdot \bm{p}\ dx, \quad \bm{p}\in \left [\mathcal{P}_{r-2}(K)\right ]^{3}.
$$
\end{itemize}

\paragraph{Space $\mathcal{A}_{r}\Lambda^{3}(\square_3)$.} Piecewise polynomials of degree $r$ in 3 variables.

\subsection{Trimmed-Adini complex}

The Adini element can also be treated as the first element in a different finite element sequence, which we call the trimmed-Adini complex. 
The shape function spaces are those of the trimmed serendipity family, defined in~\cite{GK2016} as
$$
\mathcal{S}_{r}^{-}\Lambda^{k}:=\mathcal{S}_{r-1}\Lambda^{k}+\kappa \mathcal{S}_{r-1}\Lambda^{k+1}.
$$
The operator $\kappa$ is the Koszul operator, which is discussed in detail in a finite element context in~\cite{Arnold.D;Falk.R;Winther.R.2006a}.
As shown in~\cite{GK2016}, the trimmed serendipity spaces ``nest'' in between the regular serendipity spaces via the inclusions $\mathcal{S}_{r}\Lambda^{k}\subset \mathcal{S}_{r+1}^{-}\Lambda^{k}\subset \mathcal{S}_{r+1}\Lambda^{k}$ and satisfy identities at form order 0 and $n$, given by $\mathcal{S}_r^-\Lambda^0=\mathcal{S}_r\Lambda^0$ and $\mathcal{S}_r^-\Lambda^n=\mathcal{S}_{r-1}\Lambda^n$.
The associated complexes in 2D and 3D are stated in \eqref{2d-fe-seqs} and \eqref{3d-fe-seqs}.
Figure \ref{fig:trimmed-adini2d} shows the lowest order trimmed-Adini complex in two dimensions.

%The complex
%$$
%\begin{diagram}
%0 & \rTo & \mathbb{R} & \rTo & \mathcal{S}_{r}^{-}\Lambda^{0}&\rTo^{d} & \mathcal{S}_{r}^{-}\Lambda^{1}&\rTo^{d} &\cdots& \rTo^{d} & \mathcal{S}_{r}^{-}\Lambda^{n-1} & \rTo^{d} &  \mathcal{S}_{r}^{-}\Lambda^{n}& \rTo&0,
%\end{diagram}
%$$
%is exact in contractible domains. Here the indices do not decrease as the $\mathcal{P}_{r}^{-}\Lambda^{k}$ family on simplicial mesh \cite{Arnold.D;Falk.R;Winther.R.2006a} and the tensor product elements $\mathcal{Q}_{r}^{-}\Lambda^{k}$. 

\begin{center}
\begin{figure}

\setlength{\unitlength}{1.2cm}
\begin{picture}(2,2)(-2,0)
\put(0,0){
\begin{picture}(2,2)
\put(-1, 0){\line(1,0){2}} 
\put(1, 2){\line(-1,0){2}}
\put(-1,0){\line(0,1){2}}
\put(1, 2){\line(0, -1){2}}
\put(-1,0){\circle*{0.1}}
\put(-1,0){\circle{0.2}}
\put(1.,0){\circle*{0.1}}
\put(1.,0){\circle{0.2}}
\put(-1,2){\circle*{0.1}}
\put(-1,2){\circle{0.2}}
\put(1,2){\circle*{0.1}}
\put(1,2){\circle{0.2}}
\end{picture}
}

\put(1.5, 1){\vector(1, 0){1}}
\put(1.68, 1.15){$\curl$}

\put(4,0){
\put(-1, 0){\line(1,0){2}} 
\put(1, 2){\line(-1,0){2}}
\put(-1,0){\line(0,1){2}}
\put(1, 2){\line(0, -1){2}}
\put(-0.95,0){\circle*{0.1}}
\put(-1.05,0){\circle*{0.1}}
\put(1.05,0){\circle*{0.1}}
\put(0.95,0){\circle*{0.1}}
\put(-0.95,2){\circle*{0.1}}
\put(-1.05,2){\circle*{0.1}}
\put(1.05,2){\circle*{0.1}}
\put(0.95,2){\circle*{0.1}}
\put(0, 2){\vector(0,1){0.3}}
\put(0, 0){\vector(0,-1){0.3}}
\put(-1, 1){\vector(-1,0){0.3}}
\put(1, 1){\vector(1,0){0.3}}
%\put(1, 1){\circle*{0.1}}
%\put(-0.2, 1){\circle*{0.1}}
%\put(0.2, 1){\circle*{0.1}}
\put(-0.3, 0.9){+5}
}

\put(5.5, 1){\vector(1, 0){1}}
\put(5.75, 1.1){{$\div$}}

\put(8,0){
\begin{picture}(2,2)
\put(-1, 0){\line(1,0){2}} 
\put(1, 2){\line(-1,0){2}}
\put(-1,0){\line(0,1){2}}
\put(1, 2){\line(0, -1){2}}
%\put(0, 1){\circle*{0.1}}
%\put(-0.1, 0.8){\circle*{0.1}}
%\put(0.1, 0.8){\circle*{0.1}}
\put(-0.3, 0.9){+6}
\end{picture}
}

%\put(9, 1){\vector(1, 0){1}}
\end{picture}
\caption{The lowest order trimmed-Adini complex, with regularity $H(\curl)\rightarrow {H}(\div)\rightarrow L^{2}$. The local shape function spaces are $\mathcal{S}_{3}^{-}\Lambda^{0}\rightarrow \mathcal{S}_{3}^{-}\Lambda^{1} \rightarrow \mathcal{S}_{3}^{-}\Lambda^{2}$. }
\label{fig:trimmed-adini2d}

\end{figure}
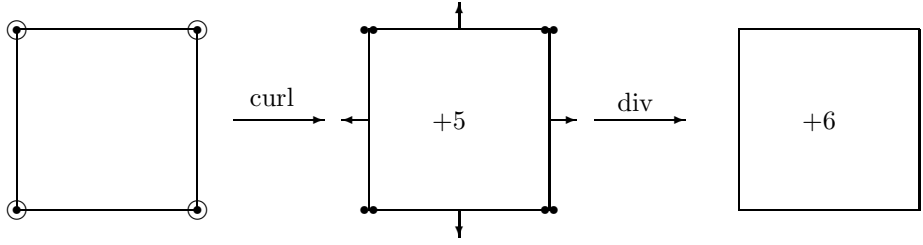
\end{center}

\subsubsection{Two dimensions}

The 2D trimmed Adini sequence is denoted by
\begin{align}\label{2d-trim-adini-complex}
\begin{diagram}
0 & \rTo & \mathbb{R} & \rTo & \mathcal{A}_{r}^{-}\Lambda^{0}(\mathcal{T}_{h}) & \rTo^{\curl} &\mathcal{A}_{r}^{-}\Lambda^{1}(\mathcal{T}_{h}) & \rTo^{\div} &\mathcal{A}_{r}^{-}\Lambda^{2}(\mathcal{T}_{h})& \rTo^{} & 0.
\end{diagram}
\end{align}
Given the identities on the serendipity and trimmed serendipity spaces for $0$-forms and $n$-forms just mentioned, the sequence could also be written as $\mathcal{A}_{r}\Lambda^{0}\rightarrow\mathcal{A}_{r}^{-}\Lambda^{1}\rightarrow\mathcal{A}_{r-1}\Lambda^{2}$.
Thus, only the space $\mathcal{A}^{-}_{r}\Lambda^{1}(\square_2)$ is distinct from a 2D Adini-type element already defined.

\paragraph{Space $\mathcal{A}^{-}_{r}\Lambda^{1}(\square_2)$.}
 The degrees of freedom can be given by
\begin{itemize}
\item
function evaluation at each vertex $x\in \mathcal{V}$:
$$
\bm{u}_{i}(x),\quad i=1, 2,
$$
\item
moments on each edge $e\in \mathcal{E}$:
$$
\int_{e}\bm{u}\cdot\bm{n}_{e}q\ dx, \quad q\in \mathcal{P}_{r-3}(e),
$$
\item
interior DoFs in $K\in \mathcal{K}$:
$$
\int_{K}\bm{u}\cdot\bm{p}\ dx, \quad \bm{p}\in \left [\mathcal{P}_{r-3}(K)\right ]^{2}\oplus \curl \mathcal{H}_{r-1}\Lambda^{0}(K),
$$
where $\mathcal{H}_{r-1}\Lambda^{0}(K)$ denotes the space of (scalar) homogeneous polynomials of degree $r-1$ on $K$. 
\end{itemize}

%Figure \ref{fig:trimmed-adini2d} shows the lowest order case.

\subsubsection{Three dimensions}

The 3D trimmed Adini complex is denoted by
\begin{align}\label{3D-adini-complex}
\begin{diagram}
0& \rTo &\mathbb{R} & \rTo &  \mathcal{A}_{r}^{-}\Lambda^{0}(\mathcal{T}_{h}) & \rTo^{\grad} & \mathcal{A}_{r}^{-}\Lambda^{1}(\mathcal{T}_{h}) & \rTo^{\curl} &\mathcal{A}_{r}^{-}\Lambda^{2}(\mathcal{T}_{h}) & \rTo^{\div} &\mathcal{A}_{r}^{-}\Lambda^{3}(\mathcal{T}_{h})& \rTo^{} & 0.
\end{diagram}
\end{align}
Using the identities for 0-forms and $n$-forms, we can re-write this sequence as
\begin{align}\label{3D-adini-complex-alt}
\begin{diagram}
0& \rTo &\mathbb{R} & \rTo &  \mathcal{A}_{r}\Lambda^{0}(\mathcal{T}_{h}) & \rTo^{\grad} & \mathcal{A}_{r}^{-}\Lambda^{1}(\mathcal{T}_{h}) & \rTo^{\curl} &\mathcal{A}_{r}^{-}\Lambda^{2}(\mathcal{T}_{h}) & \rTo^{\div} &\mathcal{A}_{r-1}\Lambda^{3}(\mathcal{T}_{h})& \rTo^{} & 0.
\end{diagram}
\end{align}
Thus, we will only describe the 1-form and 2-form spaces in 3D.

\paragraph{Space $\mathcal{A}^{-}_{r}\Lambda^{1}(\square_3)$.}
%The shape function space is $\mathcal{S}_r^-\Lambda^1(\square_3)$.
The DoFs for a vector valued function $\bm{u}$ can be given by
\begin{itemize}
\item
value at each vertex $x\in \mathcal{V}$:
$$
\bm{u}(x),
$$
\item
moments on each edge $e\in \mathcal{E}$:
$$
\int_{e}\bm{u}\cdot\bm{\tau}_{e}q\ dx, \quad q\in \mathcal{P}_{r-3}(e),
$$
\item
moments on each face $f\in \mathcal{F}$:
$$
\int_{f}\left (\bm{u}\times \bm{\nu}_{f}\right )\cdot \bm{p}\ dx, \quad p\in \left [\mathcal{P}_{r-3}(f)\right ]^{2}\oplus \grad\mathcal{H}_{r-1}\Lambda^{0}(f),
$$
where $\bm{u}\times \bm{\nu}_{f}$ is considered as a 2D vector,
\item
interior DoFs in $K\in \mathcal{K}$:
$$
\int_{K}\bm{u}\cdot\bm{s}\ dx, \quad \bm{s}\in \left [\mathcal{P}_{r-5}(K)\right ]^{3}\oplus \curl\mathcal{H}_{r-3}\Lambda^{1}(K).
$$
\end{itemize}

\paragraph{Space $\mathcal{A}^{-}_{r}\Lambda^{2}(\square_3)$.}

The DoFs for a vector function $\bm{v}$ can be taken as

\begin{itemize}
\item
moments on each face $f\in \mathcal{F}$:
$$
\int_{f}\bm{v}\cdot\bm{\nu}_{f}q\ dx,\quad q\in \mathcal{P}_{r-1}(f),
$$
\item
interior DoFs in $K\in \mathcal{K}$:
$$
\int_{K}\bm{v}\cdot \bm{p}\ dx, \quad \bm{p}\in \left [\mathcal{P}_{r-3}(K)\right ]^{3} \oplus\grad\mathcal{H}_{r-1}\Lambda^0(K).
$$
\end{itemize}

\subsection{Unisolvence and exactness}

%\textcolor{red}{This next part can be moved to later in the paper; just putting some text in here for now.}

The unisolvence of the elements and the exactness of the various complexes introduced in this section follow from the properties of the serendipity and the DoF-transfer operations shown in Lemma \ref{lem:DoF-unisolvence} -- Lemma \ref{lem:serendipity-exact}. 
For instance, since the Hermite and Adini complexes are the images of the $\mathscr{T}$ applied to the tensor product and serendipity complexes, respectively (recall the diagram in \eqref{3d-diagram-serendipity}), and since $\mathscr{T}$ preserves unisolvence by Lemma \ref{lem:DoF-unisolvence} and exactness by Lemma \ref{lem:DoF-exactness}, it follows that the Hermite and Adini complexes are are unisolvent and exact.

It is possible to prove exacness without applying the DoF-transfer to the $\mathcal{S}_{\bs}\Lambda^{\bs}$ sequence. 
Consider the diagram shown in \eqref{3d-diagram-serendipity2}, which does not assume any relation between $\mathcal{S}_{\bs}\Lambda^{\bs}$ and $\mathcal{A}_{\bs}\Lambda^{\bs}$.
\begin{equation}\label{3d-diagram-serendipity2}
\begin{tikzcd}
& \ker(\mathscr{D}_{0}, \mathcal{Q})\arrow[dl, "="] \arrow[rr, "\curl"] \arrow[dd, "\iota"]& & \ker(\mathscr{D}_{1},\mathcal{Q})\arrow[dl, "="]\arrow[rr, "\div"]  \arrow[dd, "\iota"] & & \ker(\mathscr{D}_{2},\mathcal{Q})\arrow[dl, "="]\arrow[dd, "\iota"] \\
 \ker(\mathscr{D}_{0}, \mathcal{G}) \arrow[rr, crossing over, "\curl"] \arrow[dd, crossing over, "\iota"] & & \ker(\mathscr{D}_{0}, \mathcal{G})  \arrow[rr, crossing over, "\div"] \arrow[dd, crossing over, "\iota"]  & &  \ker(\mathscr{D}_{0}, \mathcal{G}) \arrow[dd, crossing over, "\iota"] \\
& \mathcal{Q}_{r+2}^-\Lambda^{0} \arrow[dl, "\mathscr{T}"] \arrow[rr, "\curl"] \arrow[dd, "\mathscr{Q}"]& & \mathcal{Q}_{r+2}^-\Lambda^{1} \arrow[dl, "\mathscr{T}"]\arrow[rr, "\div"]  \arrow[dd, "\mathscr{Q}"] & & \mathcal{Q}_{r+1}^-\Lambda^{2} \arrow[dl, "\mathscr{T}"]\arrow[dd, "\mathscr{Q}"] \\
 \mathcal{G}^{-}_{r+2}\Lambda^{0} \arrow[rr, crossing over, "\curl"] \arrow[dd, crossing over, "\mathscr{Q}"] & & \mathcal{G}^{-}_{r+2}\Lambda^{1}   \arrow[rr, crossing over, "\div"]  & &  \mathcal{G}^{-}_{r+1}\Lambda^{2} \\
& \mathcal{S}_{r+2}\Lambda^{0}  \arrow[rr] &  & \mathcal{S}_{r+1}\Lambda^{1}   \arrow[rr]  & & \mathcal{S}_{r}\Lambda^{2}  \\
%& \mathcal{S}_{r+2}\Lambda^{0} \arrow[dl, "\mathscr{T}"] \arrow[rr] &  & \mathcal{S}_{r+1}\Lambda^{1} \arrow[dl, "\mathscr{T}"]  \arrow[rr]  & & \mathcal{S}_{r}\Lambda^{2} \arrow[dl, "\mathscr{T}"] \\
 \mathcal{A}_{r+2}\Lambda^{0} \arrow[rr, "\curl"] & &  \mathcal{A}_{r+1}\Lambda^{1}  \arrow[uu, leftarrow, crossing over, "\mathscr{Q}"]  \arrow[rr, "\div"]  & &  \mathcal{A}_{r}\Lambda^{2} \arrow[uu, leftarrow, crossing over, "\mathscr{Q}"]\\
\end{tikzcd}
\end{equation}
By Lemma \ref{lem:DoF-unisolvence} and Lemma \ref{lem:DoF-exactness},  the Hermite family $\mathcal{G}^{-}_{\bs}\Lambda^{\bs}$ is unisolvent and exact. Since the serendipity family $\mathcal{S}_{\bs}\Lambda^{\bs}$ is known to be exact on contractible domains~\cite{arnold2014finite}, we know that the kernel sequence 
$$
\begin{diagram}
0 & \rTo &  \ker(\mathscr{D}_{0}, \mathcal{Q})& \rTo^{\curl} &  \ker(\mathscr{D}_{1}, \mathcal{Q})& \rTo^{\div} &  \ker(\mathscr{D}_{2}, \mathcal{Q}) & \rTo& 0,
\end{diagram}
$$
is exact by Lemma \ref{lem:serendipity-exact}. Since the DoF-transfer does not change bubble spaces (by Lemma \ref{lem:serendipity-bubbles}) the kernel sequences are in fact equal, i.e.\ $\ker (\mathscr{D}_{\bs}, \mathcal{Q})=\ker (\mathscr{D}_{\bs}, \mathcal{G})$.  Using Lemma \ref{lem:serendipity-exact} again, we conclucde the Adini sequence $\mathcal{A}_{\bs}\Lambda^{\bs}$ is also exact. 

\section{Approximation and convergence properties}\label{sec:convergence}

The DoF-transfer and serendipity operations aid in establishing approximation and convergence properties. The DoF-transfer operation ensures a discrete Korn inequality and  non-conforming  convergence. The serendipity operation preserves these properties since it is a local space reduction that does not change inter-element continuity.

%%%%%%%%%%%%%%%%%%%%%%%%%%%%%%%%%%%%%%
\subsection{A discrete Korn inequality}\label{sec:korn}

The Korn inequality is an indispensable tool in linear elasticity, asserting that if a vector field is orthogonal to the rigid body motion, then its $H^{1}$ norm can be controlled by its symmetric gradient. Specifically, the Korn inequality reads: if $u$ and $v$ are vector valued functions in a domain $\Omega$, then there exists a positive constant $C>0$ such that
$$
\|u\|_{1}^{2}\leq C\left ( \left \|\epsilon(u)  \right \|^{2}+\|u\|^{2}\right ), \quad\forall u\in \left [H^1(\Omega) \right ]^{n},
$$ 
$$
\|v\|_{1}^{2}\leq C\left (\|\epsilon(v)\|^{2}+\sum_{i=1}^{n}\int_{\partial\Omega}\left |v_{i}\right |\, ds\right ), \quad \forall v\in \left [H^1(\Omega) \right ]^{n},
$$
and
$$
\|v\|_{1}^{2}\leq C\|\epsilon(v)\|^{2}, \quad \forall v\in \left [H^1_{0}(\Omega) \right ]^{n}. 
$$
Here $\epsilon$ is the symmetric gradient mapping a vector to a symmetric matrix, defined by $[\epsilon(u)]_{ij}:=1/2 (\partial_{i}u_{j}+\partial_{j}u_{i})$.

For  non-conforming finite element spaces, a discrete version of Korn's inequalities is desirable. In this case, the derivatives appearing in the original Korn's inequality are replaced by the piecewise derivatives and the broken Sobolev norms
$$
\|u\|_{m, h}:=\left (\sum_{T\in \mathcal{T}_{h}}\|u\|_{m, T}^{2} \right )^{1/2}.
$$

We first recall the results in \cite{ming1994generalized}. Following the general results, we will see that the DoF-transfer operation leads to finite elements satisfying the discrete Korn inequality. 
The following two conditions for a finite dimensional space $V_{h}$ are crucial for the discrete Korn inequality in \cite{ming1994generalized}.
\begin{itemize}
\item[(H1)]  There exists an integer $r\geq 1$ such that for any $v\in V_{h}$, $\left . v\right |_{T}\in \mathcal{P}_{r}(T), ~\forall T\in \mathcal{T}_{h}$.
\item[(H2)] For  any $v\in V_{h}$, $T\in \mathcal{T}_{h}$, let $F$ be an arbitrary $(n-1)$ dimensional face of $T$.  Then $v$ is continuous on a set consisting of at least $n$ points that are not on a common $n-2$ dimensional hyperplane, and the set is affine invariant.
\end{itemize}

\begin{theorem}[\cite{ming1994generalized}]\label{thm:korn}
Assume that the assumptions (H1) and (H2) are true. Then there exists a constant $C$ independent of the mesh size $h$, such that the following discrete Korn inequalities hold:
$$
\|v\|_{1, h}^{2}\leq C\left ( \sum_{T\in \mathcal{T}_{h}}\|\epsilon(v)\|_{0, T}^{2}+\|v\|_{0, \Omega}^{2}\right ), \quad \forall v\in V_{h},
$$
$$
\|v\|_{1, h}^{2}\leq C\left ( \sum_{T\in \mathcal{T}_{h}}\|\epsilon(v)\|_{0, T}^{2}+\sum_{i=1}^{n}\int_{\partial\Omega}\left |v_{i}\right |\, ds\right ), \quad \forall v\in V_{h},
$$
and
$$
\|v\|_{1, h}^{2}\leq C\sum_{T\in \mathcal{T}_{h}}\|\epsilon(v)\|_{0, T}^{2}, \quad \forall v\in V_{h}^{0},
$$
where $V_{h}^{0}$ is the subspace of $V_{h}$ with vanishing degrees of freedom on the boundary.
\end{theorem}

In some cases, the condition (H2) is a strong assumption. For example, generally the shape functions of the Crouzeix-Raviart  non-conforming  element \cite{brenner2015forty} are only continuous at $n-1$ points on a $n-1$ dimensional face. However, with the DoF-transfer operation, the resulting finite element functions are continuous at each vertex. On an $n-1$ dimensional face, this amounts to at least $n$ points. Therefore the assumption (H2) is fulfilled and we conclude the following result.
\begin{theorem}
The 0- and 1-forms in the Hermite, Adini and trimmed-Adini complexes satisfy the discrete Korn inequalities in Theorem \ref{thm:korn}.
\end{theorem}

%%%%%%%%%%%%%%%%%%%%%%%%%%%%%%%%%%%%%%
\subsection{Convergence as  non-conforming  elements}

The DoF-transfer operation ensures extra continuity at vertices compared to standard elements, namely, $C^{1}$ continuity for 0-forms and $C^0$ continuity for 1-forms. This order of continuity allows provable convergence results for problems requiring more regularity in the solution.  Specifically, the resulting 0-forms can be regarded as {convergent}  non-conforming  elements for the biharmonic equations \cite{adini,hu2016capacity} while the 1-forms are convergent  non-conforming  elements for Poisson type equations involving the scalar Laplacian operator.

First, recall the Strang lemma (c.f. \cite{brenner2007mathematical}).
\begin{lemma}[Strang Lemma]
Let $V$ and $V_{h}$ be Hilbert spaces and  $\dim V_{h}<\infty$. Let $a_{h}(\cdot, \cdot)$ be a symmetric positive-definite bilinear form on $V+V_{h}$ that reduces to $a(\cdot, \cdot)$ on $V$. Let $u\in V$ solve
$$
a(u, v)=F(v), \quad\forall v\in V, 
$$
where $F\in V^{\ast}\cap V_{h}^{\ast}$. Let $u_{h}\in V_{h}$ solve
$$
a_{h}(u_{h}, v)=F(v), \quad\forall v\in V_{h}.
$$
Then
\begin{align}\label{strang}
\left \|u-u_{h}\right \|_{h}\leq \inf_{v_{h}\in V_{h}}\|u-v_{h}\|_{h}+\sup_{w_{h}\in V_{h}\backslash \{0\}}\frac{\left |a_{h}(u-u_{h}, w_{h})\right |}{\|w_{h}\|_{h}}.
\end{align}
\end{lemma}

The first term on the right hand side of \eqref{strang} is the local approximation error, which is determined by the local shape function spaces. The second term, which does not appear in conforming methods, is the consistency error. The consistency error is generally determined by the inter-element continuity, although not strong enough to guarantee conformity. As we shall see, for the new elements proposed in this paper, the consistency error is controlled also because of the proper local shape function spaces and the geometric symmetry.  In particular, we note that if the consistency error of a space $V_{h}$ tends to zero asymptotically as $h\rightarrow 0$, then any of its subspaces inherit this property. This means that the space reductions, either by the serendipity or by the DoF-transfer operations, do not break the convergence of  non-conforming  methods. 

In this section, we prove the convergence of the DoF-transferred 1-forms for Poisson equations.
Since one often seeks simple nonconforming elements and one cannot improve the  consistency error by increasing local polynomial degrees, we focus on  the lowest order Adini 1-form $\mathcal{A}_{2}\Lambda^{1}$. The space $\mathcal{A}_{2}\Lambda^{1}$ is not a tensor product of copies of scalar elements, meaning that one cannot collect the shape functions and degrees of freedom of one component to get a unisolvent scalar element. However, one can obtain a scalar element with the same inter-element continuity by adding some interior degrees of freedom. We now show that each component of $\mathcal{A}_{2}\Lambda^{1}$ (in the above sense) yields a convergent scalar element for the Poisson problem. Then the convergence of each component implies that of  the  element $\mathcal{A}_{2}\Lambda^{1}$  for  vector Poisson problems.
We start the analysis with the following explicit form of the bases.%
\begin{lemma}
The shape function space of the Adini 1-form $\mathcal{A}_{2}\Lambda^{1}$ has a direct sum decomposition
$$
\mathcal{A}_{2}\Lambda^{1}=\mathcal{P}_{2}\Lambda^{1}\oplus \mathcal{J}_{2}\Lambda^{1}\oplus d\mathcal{J}_{3}\Lambda^{0},
$$
where in three space dimensions
$$
\mathcal{J}_{2}\Lambda^{1}=\mathrm{span}\left \{\left (\begin{array}{c}
-xyz\\ x^{2}z\\0
\end{array}
\right ),
\left (\begin{array}{c}
-y^{2}z\\ xyz\\0
\end{array}
\right ),
\left (\begin{array}{c}
0\\ -xyz\\xy^{2}
\end{array}
\right ),
\left (\begin{array}{c}
0\\ -xz^{2}\\xyz
\end{array}
\right ),
\left (\begin{array}{c}
-xyz\\ 0\\x^{2}y
\end{array}
\right ),
\left (\begin{array}{c}
-yz^{2}\\ 0\\xyz
\end{array}
\right )
\right \},
$$
and
\begin{align*}
d\mathcal{J}_{3}\Lambda^{0}=\mathrm{span}&\left \{\left (\begin{array}{c}
3x^{2}y\\ x^{3}\\0
\end{array}
\right ),
\left (\begin{array}{c}
3x^{2}z\\0\\x^{3}
\end{array}
\right ),
\left (\begin{array}{c}
2xyz\\ x^{2}z\\x^{2}y
\end{array}
\right ),
\left (\begin{array}{c}
3x^{2}yz\\ x^{3}z\\x^{3}y
\end{array}
\right ),
\left (\begin{array}{c}
0\\ 3y^{2}z\\y^{3}
\end{array}
\right ),
\left (\begin{array}{c}
y^{3}\\ 3xy^{2}\\0
\end{array}
\right )
\right. ,\\
\quad&\left.\left (\begin{array}{c}
y^{2}z\\2xyz\\xy^{2}
\end{array}
\right ),
\left (\begin{array}{c}
y^{3}z\\3xy^{2}z\\xy^{3}
\end{array}
\right ),
\left (\begin{array}{c}
0\\z^{3}\\3yz^{2}
\end{array}
\right ),
\left (\begin{array}{c}
z^{3}\\0\\3xz^{2}
\end{array}
\right ),
\left (\begin{array}{c}
yz^{2}\\xz^{2}\\2xyz
\end{array}
\right ),
\left (\begin{array}{c}
yz^{3}\\xz^{3}\\3xyz^{2}
\end{array}
\right )\right \}.
\end{align*}
\end{lemma}

%Corresponding to the x-direction component of the finite element, we present the shape functions on this reference cubic. Denote 
Without loss of generality, we focus on the details of the $x$-component of the Adini 1-forms in 3D on a reference cubic $K=(-1,1)^3$.  We use ${\mathcal{E}_K^x}$ to denote the edges of $K$ parallel to the $x$-axis, and ${\mathcal{F}_K^x}$ for the faces of $K$ parallel to the $x$-axis. 

%(the shape function space is affine-equivalent. \textcolor{blue}{(should we emphasize this? --Kaibo)}). 
Collecting the shape functions of the x-component, we define
$$
P_K^x:=\mathcal{P}_2(K)+{\rm span}\{xyz,x^2y,x^2z,yz^2,y^2z,y^3,z^3,zy^3,z^3y,x^2yz\}.
$$
Equivalently, 
$$
P_K^x={\rm span}\{1,x,x^2\}\{1,y\}\{1,z\}+{\rm span}\{y^2,z^2,y^2z,yz^2,y^3,z^3,y^3z,yz^3\},
$$
where the three sets of curly braces indicate one choice is to be made from each set of braces before multiplying together.  
The following degrees of freedom define a unisolvent finite element with local shape function space $P_K^x$:
\begin{itemize}
\item
function evaluation at each vertex $x\in \mathcal{V}$:
$$
v(x),
$$
\item
moments on edges parallel to the $x$-axis:
$$
\int_{e}v \, dx, \quad \forall e\in \mathcal{E}_K^x,
$$
\item
moments on faces parallel to the $x$-axis:
$$
\int_{f}v \, dx, \quad \forall f\in \mathcal{F}_K^x,
$$
\item  interior degrees of freedom on $K$:
$$
\int_{K}vq\, dx, \quad q\in P_{K}^{x}\cap H^{1}_{0}(K).
$$
\end{itemize}
We note that the above degrees of freedom on faces (including vertices and edges) are those for the $x$-component of $\mathcal{A}_{2}\Lambda^{1}$, which are linearly independent for $P_{K}^{x}$ and have dimension 16. On the other hand, the space $P_{K}^{x}$ has dimension 20. This difference is due to the non-tensor-product nature of $\mathcal{A}_{2}\Lambda^{1}$ as explained above. Therefore we also include the interior degrees of freedom with dimension $20-16=4$. 
\begin{lemma}
The above degrees of freedom are unisolvent for $\mathcal{A}_{2}\Lambda^{1}$.
\end{lemma}
\begin{comment}
\begin{proof}
Both the degrees of freedom and the $P_{K}^{x}$ space has dimension 20. Therefore it suffices to show that if all the degrees of freedom vanish on $v\in P_{K}^{x}$, then $v=0$. Actually, from the conformity of $\mathcal{A}_{2}\Lambda^{1}$, we see that the vanishing of degrees of freedom on faces (including vertices and edges) implies that $v=0$ on all the faces. 
\end{proof}
\end{comment}

The local shape function space $P_{K}^{x}$ and the above degrees of freedom give a finite element space on a mesh $\mathcal{T}_{h}$ on $\Omega$, which we denote as $V_{h}$. Define $V_{h0}$ as the corresponding space with homogeneous boundary conditions:
$$
V_{h0}:=\left\{v\in V_h: v(a)=0,\ a\in \mathcal{V}^{b};\ \int_e v=0,\ e\in\mathcal{E}^{xb},\ \int_fv=0,\ f\in\mathcal{F}^{xb}\right \},
$$
where $\mathcal{V}^{b}$ is the set of boundary vertices, $\mathcal{E}^{xb}$ the boundary edges parallel to the $x$-axis and $\mathcal{F}^{xb}$ the boundary faces parallel to the $x$-axis.

Consider the variational form of the Poisson equation: given $f\in L^{2}(\Omega)$, find $u\in H^1_0(\Omega)$, such that 
\begin{equation}\label{eq:poisson}
(\nabla u,\nabla v)=(f,v),\quad\forall\,v\in H^1_0(\Omega);
\end{equation}
and its  non-conforming  element discretization: find $u_h\in V_{h0}$, such that 
\begin{equation}\label{eq:poissondis}
(\nabla_hu_h,\nabla_hv_h)=(f,v_h),\quad\forall\,v_h\in V_{h0}.
\end{equation}
Here $\nabla_{h}$ is the piecewise gradient defined on $\mathcal{T}_{h}$.

To show the convergence of \eqref{eq:poissondis}, we introduce some auxiliary spaces, again using the convention that a list inside curly braces indicates one of the options should be chosen:
$$
B^x_K={\rm span}\left\{(x^2-1)\left[(y^2-1)\{1,x\}+\{1,y\}(z^2-1)\right]\right\},
$$
$$
Q^x_K={\rm span}\{1,x,x^2\}\{1,y\}\{1,z\};
$$
and 
$$
A^x_K=Q^x_K+B^x_K.
$$
The auxiliary space $A^x_{K}$ is defined so that it is unisolvent with the vertex, edge and faces DoFs of $V_{h}$, and also forms a  $C^{0}$-conforming global space.
Denote by $\Pi_Q$ $\Pi_B$ and $\Pi_A$ the interpolations such that
$$
%\Pi_Qv\in Q^x_K;\ (\Pi_Qv)(a)=v(a),\ a\mbox{\ any\ vertex\ of}\ K;\ \int_{e}(\Pi_Qv)=\int_{e}v,\ e\in\mathcal{E}^x_K;
\Pi_Qv\in Q^x_K;\ (\Pi_Qv)(a)=v(a),\ \forall a\in \mathcal{V}(K);\ \int_{e}(\Pi_Qv)=\int_{e}v,\ \forall e\in\mathcal{E}^x_K;
$$
$$
\Pi_Bv\in B^x_K; \ \int_{f}(\Pi_Bv)=\int_{f}v,\ \forall f\in\mathcal{F}^x_K,
$$
and 
$$
\Pi_Av\in A^x_K;\ (\Pi_Av)(a)=v(a);\ \int_{e}(\Pi_Av)=\int_{e}v,\ \forall e\in \mathcal{E}^x_K;\ \int_{F}(\Pi_Av)=\int_{f}v,\ \forall f\in\mathcal{F}^x_K.
$$
%\begin{lemma}
%\end{lemma}
{It is straightforward to check that the interpolations are all well-defined. }

\begin{lemma}\label{lem:ABQ}
For any $u\in C^{0}(K)$, it holds that
$$
(\Pi_A-\Pi_B)(I-\Pi_Q)u=0.
$$
\end{lemma}
\begin{proof}
For any vertex $a$ of $K$, $[(I-\Pi_Q)u](a)=0$, and therefore $\Pi_A(I-\Pi_Q)u\in B^x_K$; actually, $\{v\in A_K:v(a)=0\}=B^x_K$. The assertion follows. 
\end{proof}

\begin{lemma}\label{lem:localstruc}
For any $v\in C^{0}(K)$, it holds that
$\displaystyle\int_{\partial K}(v-\Pi_Av)\, dx=0$,\ \ $v\in P^x_K$.
\end{lemma}

\begin{proof}
Note that $\Pi_A\Pi_Q=\Pi_Q$, and we have by Lemma \ref{lem:ABQ} that
$$
u-\Pi_Au=(u-\Pi_Qu)-\Pi_A(u-\Pi_Qu)=(u-\Pi_Qu)-\Pi_B(u-\Pi_Qu)=(I-\Pi_B)(I-\Pi_Q)u.
$$
By direct calculations, we get Table \ref{tab:projection}.
\begin{table}[htbp]
\begin{tabular}{|c||c|c|c|c|c|c|c|c|}
\hline
$u$ &$y^2$&$z^2$&$y^3$&$y^2z$&$yz^2$&$z^3$&$y^3z$&$yz^3$
\\
\hline
$\Pi_Qu$&1&1&$y$&$z$&$y$&$z$&$yz$&$yz$
\\
\hline
$u-\Pi_Qu$ &$y^2-1$&$z^2-1$&$y^3-y$&$y^2z-z$&$yz^2-y$&$z^3-z$&$y^3z-yz$&$yz^3-yz$
\\
\hline
\end{tabular}
\caption{Interpolations of shape functions.}
\label{tab:projection}
\end{table}
%\begin{table}[htbp]
%\begin{tabular}{ccc}
%\hline
%$u$ &$\Pi_Qu$& $u-\Pi_Ku$
%\\
%\hline
%$y^2$ &1&
%\\
%\hline
%$z^2$,&1&
%\\
%\hline
%$y^2z$,&$z$&
%\\
%\hline
%$yz^2$,&$y$&
%\\
%\hline
%$y^3$,&$y$&
%\\
%\hline
%$z^3$,&$z$&
%\\
%\hline
%$y^3z$,&$yz$&
%\\
%\hline
%$yz^3$&$yz$
%&
%\\
%\hline
%\end{tabular}
%\caption{$\Pi_Pu=...$ $\Pi_Ku=\Pi_Pu+(\Pi_Ku-\Pi_Pu)=\Pi_Pu+\Pi_K(u-\Pi_Pu)$, $u-\Pi_Ku=(u-\Pi_P)-\Pi_K(u-\Pi_Pu)=(u-\Pi_Pu)-\Pi_F(u-\Pi_Pu)=(I-\Pi_F)(I-\Pi_P)u$}
%\end{table}

Denote faces by  (Left, Right, Out, Interior, Bottom, Top)
$$
\left\{
\begin{array}{lll}
F_L\sim (-1,y,z);& F_R\sim (1,y,z);& (y,z)\in (-1,1)^2
\\
F_O\sim (x,-1,z);&F_I\sim (x,1,z);& (x,z)\in (-1,1)^2
\\
F_B\sim (x,y,-1);&F_T\sim (x,y,1);& \ (x,y)\in (-1,1)^2.
\end{array}
\right.
$$
Here, for instance, $F_L\sim (-1,y,z)$ means the face $\{x=1, y, z\in (-1, 1)\}$.

For any $v\in P^x_K$, it can be verified that
\begin{equation}
\int_{F_L}u-\Pi_Au=\int_{F_R}u-\Pi_Au,\ \ \int_{F_I}u-\Pi_Au=\int_{F_O}u-\Pi_Au,\ \ \int_{F_T}u-\Pi_Au=\int_{F_B}u-\Pi_Au.
\end{equation}
Actually, 
\begin{enumerate}
\item if $v\in Q^x_K$, then $(I-\Pi_A)v=(I-\Pi_B)(I-\Pi_Q)v=0$;
\item if $v\in {\rm span}\{y^2,z^2,y^3,y^2z,yz^2,z^3,y^3z,yz^3\}$
\begin{enumerate}
\item on $\{F_I,F_O,F_T,F_B\}$, by definition of $\Pi_B$, $\displaystyle\int_{F_I,F_O,F_T,F_B}(I-\Pi_B)((I-\Pi_Q)v)=0$;
\item on $\{F_L,F_R\}$, $\displaystyle\int_{F_L}(I-\Pi_Q)v=\displaystyle\int_{F_R}(I-\Pi_Q)v$, for $v\in{\rm span}\{y^2,z^2,y^3,y^2z,yz^2,z^3,y^3z,yz^3\}$ and $\displaystyle\int_{F_L,F_R}w=0$ for $w\in B_K$, thus $\displaystyle\int_{F_L}(I-\Pi_B)((I-\Pi_Q)v)=\int_{F_R}(I-\Pi_B)((I-\Pi_Q)v)$.
\end{enumerate}
\end{enumerate}
The assertion of the lemma then follows. 
\end{proof}

\begin{comment}
Define finite element space $V_h$ by:
\begin{itemize}
\item piecewise $P_K$
\item globally continuous: evaluation at vertices; integral on edges along $x$ direction; integral on faces parallel $x$ direction. 
\end{itemize}
$$
V_{h0}:=\{v\in V_x: v(a)=0,\ a\in \mathcal{V}^{b};\ \int_e v=0,\ e\in\mathcal{E}^{xb},\ \int_fv=0,\ f\in\mathcal{F}^{xb}\},
$$
where $\mathcal{V}^{b}$ denotes vertices on the boundary $\partial \Omega$, $\mathcal{E}^{xb}$ denotes the edges on $\partial\Omega$ with the $x$-direction and $\mathcal{F}^{xb}$ is for the boundary faces orthogonal to the $x$-axis. 
\end{comment}

Define an auxiliary space $A_{h}$ with the local shape function space $A_K^{x}$ and the same vertex, edge and face degrees of freedom as $V_{h}$. 
%Let $A_{h,0}$ be the corresponding space with homogeneous boundary condition. 
Define the nodal interpolation $\Pi_h^A$ with respect to $A_h$. We use $\llbracket u \rrbracket_f$ to denote the jump of $u$ on a face, i.e. if $f=T_{1}\cap T_{2}$, then $\llbracket u \rrbracket_{f}:=u_{T_{1}}-u_{T_{2}}$. When the context is clear, we omit the subscript $f$ and just write $\llbracket u \rrbracket$ instead.

\begin{theorem}\label{thm:convergence-poisson}
Let $u\in H^1_0(\Omega)\cap H^2(\Omega)$ and $u_h$ be the solutions of \eqref{eq:poisson} and \eqref{eq:poissondis}, respectively. Then
\begin{equation}
\|u-u_h\|_{1,h}\leqslant Ch\|u\|_{2,\Omega}.
\end{equation}  
\end{theorem}
\begin{proof}
By the Strang lemma, we can estimate the consistency error. We have
\begin{multline*}
(\nabla u,\nabla_hv_h)-(f,v_h)=\sum_{f\in \mathcal{F}}\int_f\partial_nu\llbracket v_h\rrbracket=\sum_{f\in \mathcal{F}} \int_f\partial_nu\llbracket v_h-\Pi_h^A v_h\rrbracket=\sum_{K\in{\mathcal{T}_h}}\int_{\partial K}\partial_nu(v_h-\Pi_h^Av_h)
\\
=\sum_{K\in{\mathcal{T}_h}}\sum_{i=1}^3\int_{\partial K}\partial_{x_i}u(v_h-\Pi_h^Av_h)n_i=\sum_{K\in{\mathcal{T}_h}}\sum_{i=1}^3\int_{\partial K}(\partial_{x_i}u-C_K^i)(v_h-\Pi_h^Av_h)n_i,
\end{multline*}
where the second equality is due to the $C^{0}$ conformity of $A_{h}$ and the last is by Lemma \ref{lem:localstruc}. Here $C_K^i$ is an arbitrary constant.

This way, 
$$
|(\nabla u,\nabla_hv_h)-(f,v_h)|\leqslant \sum_{K\in{\mathcal{T}_h}}\sum_{i=1}^3\inf_{C^i_K\in\mathbb{R}^{\#({\mathcal{T}_h})\times 3}}(\|\partial_{x_i}u-C^i_K\|_{\partial K}\|(v_h-\Pi_h^Av_h)\|_{\partial K})\leqslant Ch\|u\|_{2,\Omega}\|v_h\|_{1,h},
$$
where $\#({\mathcal{T}_h})$ is the number of elements in $\mathcal{T}_{h}$.
The consistency error is controlled, and the assertion follows. 
\end{proof}

Now we have proved that the space $V_{h}$, which consists of the local space $P_{K}^{x}$ and the same inter-element continuity as the $x$-component of $\mathcal{A}_{2}\Lambda^{1}$, converges as a non-conforming element for the Poisson equation. This means that each component (in a proper sense) of $\mathcal{A}_{2}\Lambda^{1}$ converges as a nonconforming element. However, since we have added interior degrees of freedom to $V_{h}$, there is one more step to conclude with the convergence of $\mathcal{A}_{2}\Lambda^{1}$ as a non-conforming element for the vector Poisson equation.

Consider the vector Poisson equation: given $\bm{f}\in [L^{2}(\Omega)]^{3}$, find $\bm u\in \left [ H^1_0(\Omega)\right ]^{3}$, such that 
\begin{equation}\label{eq:poisson-vec}
(\nabla \bm u,\nabla \bm v)=(\bm f,\bm v),\quad\forall\,\bm v\in \left [H^1_0(\Omega)\right ]^{3};
\end{equation}
and its  non-conforming  element discretization: find $\bm u_h\in \0{\mathcal{A}}_{2}\Lambda^{1}$, such that 
\begin{equation}\label{eq:poissondis-vec}
(\nabla_h\bm u_h,\nabla_h\bm v_h)=(\bm f,\bm v_h),\quad\forall\,v_h\in \0{\mathcal{A}}_{2}\Lambda^{1}.
\end{equation}
Here $\0{\mathcal{A}}_{2}\Lambda^{1}$ is the subspace ${\mathcal{A}}_{2}\Lambda^{1}$ with vanishing boundary degrees of freedom.
\begin{theorem}
Let $\bm u\in [H^1_0(\Omega)\cap H^2(\Omega)]^{3}$ and $\bm u_h$ be the solutions of \eqref{eq:poisson-vec} and \eqref{eq:poissondis-vec}, respectively. Then
\begin{equation}
\|\bm u-\bm u_h\|_{1,h}\leqslant Ch\|\bm u\|_{2,\Omega}.
\end{equation}  
\begin{proof}
The theorem follows from the Strang lemma. In fact, the local approximation error satisfies the above estimate since piecewise constants are contained in ${\mathcal{A}}_{2}\Lambda^{1}$.  Moreover, the consistency error also satisfies the estimate as we have shown for each component of  ${\mathcal{A}}_{2}\Lambda^{1}$ with the same local spaces and inter-element continuity.
\end{proof}
\end{theorem}

\begin{remark}
The two dimensional results follow the same way. 
\end{remark}

\section{A minimal complex in two dimensions}\label{sec:minimal}

By the Strang lemma, the errors of solving PDEs using  non-conforming  elements come from two sources: one is the local approximation error and another is the consistency error due to the nonconformity.  In most cases, the consistency error renders only first order convergence. In this respect, higher local polynomial degree does not increase the convergence order of  non-conforming  elements; and {\it serendipity operations for non-conforming elements should be defined not to preserve the full polynomial degree, but to preserve the constants}. From this point of view, we can further simplify the 2D Adini complex by a new serendipity operation. The construction yields a new complex which is conforming for $H^{1}$-$H(\div)$-$L^{2}$, convergent as  non-conforming  elements for $H^{2}$-$[H^{1}]^{2}$-$L^{2}$, and minimal in the sense that no interior bubbles are involved.

\begin{center}
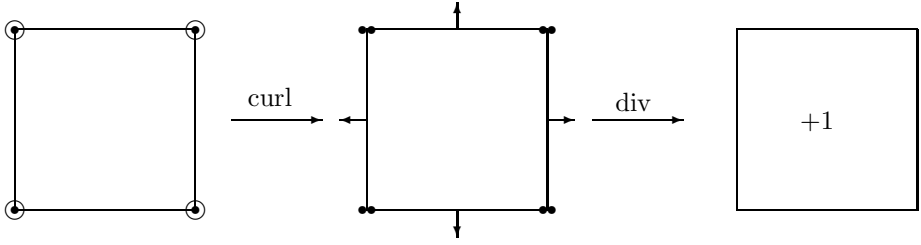
\begin{figure}

\setlength{\unitlength}{1.2cm}
\begin{picture}(2,2)(-2,0)
\put(0,0){
\begin{picture}(2,2)
\put(-1, 0){\line(1,0){2}} 
\put(1, 2){\line(-1,0){2}}
\put(-1,0){\line(0,1){2}}
\put(1, 2){\line(0, -1){2}}
\put(-1,0){\circle*{0.1}}
\put(-1,0){\circle{0.2}}
\put(1.,0){\circle*{0.1}}
\put(1.,0){\circle{0.2}}
\put(-1,2){\circle*{0.1}}
\put(-1,2){\circle{0.2}}
\put(1,2){\circle*{0.1}}
\put(1,2){\circle{0.2}}
\end{picture}
}

\put(1.5, 1){\vector(1, 0){1}}
\put(1.68, 1.15){$\curl$}

\put(4,0){
\put(-1, 0){\line(1,0){2}} 
\put(1, 2){\line(-1,0){2}}
\put(-1,0){\line(0,1){2}}
\put(1, 2){\line(0, -1){2}}
\put(-0.95,0){\circle*{0.1}}
\put(-1.05,0){\circle*{0.1}}
\put(1.05,0){\circle*{0.1}}
\put(0.95,0){\circle*{0.1}}
\put(-0.95,2){\circle*{0.1}}
\put(-1.05,2){\circle*{0.1}}
\put(1.05,2){\circle*{0.1}}
\put(0.95,2){\circle*{0.1}}
\put(0, 2){\vector(0,1){0.3}}
\put(0, 0){\vector(0,-1){0.3}}
\put(-1, 1){\vector(-1,0){0.3}}
\put(1, 1){\vector(1,0){0.3}}
%\put(1, 1){\circle*{0.1}}
%\put(-0.2, 1){\circle*{0.1}}
%\put(0.2, 1){\circle*{0.1}}
\put(-0.3, 0.9){}
}

\put(5.5, 1){\vector(1, 0){1}}
\put(5.75, 1.1){{$\div$}}

\put(8,0){
\begin{picture}(2,2)
\put(-1, 0){\line(1,0){2}} 
\put(1, 2){\line(-1,0){2}}
\put(-1,0){\line(0,1){2}}
\put(1, 2){\line(0, -1){2}}
%\put(0, 1){\circle*{0.1}}
%\put(-0.1, 0.8){\circle*{0.1}}
%\put(0.1, 0.8){\circle*{0.1}}
\put(-0.3, 0.9){+1}
\end{picture}
}

%\put(9, 1){\vector(1, 0){1}}
\end{picture}
\caption{The reduced Adini complex. }
\label{fig:adini2d-reduced}

\end{figure}
\end{center}

%Although reducing interior DoFs is not a trivial task, we can still reduce, for example, the Adini-de Rham family with lowest degree to obtain a simpler family with all desired properties.   
We define the reduced space $\tilde{\mathcal{A}}_{2}\Lambda^{1}$ by 
\begin{align*}
\tilde{\mathcal{A}}_{2}\Lambda^{1}(K):&=\left[\mathcal{P}_{1}(K)\right ]^{2}+\mathrm{span}\left \{\curl(x^{3}), \curl (y^{3}) \curl(x^{2}y), \curl(xy^{2}), \curl(x^{3}y), \curl (xy^{3}) \right \}\\
&=\curl \mathcal{A}_{3}\Lambda^{0}(K)+\mathrm{span}\{(x, 0)^{T}, (0, y)^{T}\},
\end{align*}
and define $\tilde{\mathcal{A}}_{1}\Lambda^{2}$ by
$$
\tilde{\mathcal{A}}_{1}\Lambda^{2}(K):=\mathcal{P}_{0}(K).
$$
The degrees of freedom for $\bm{u}\in \tilde{\mathcal{A}}_{2}\Lambda^{1}$, as shown in Figure  \ref{fig:adini2d-reduced},  are given by
\begin{itemize}
\item
function evaluation at each vertex $x\in \mathcal{V}$,
\item
moments on each edge $e\in \mathcal{E}$:
$$
\int_{e}\bm{u}\cdot \bm{\nu}_{e}\, dx.
$$
\end{itemize}
The degrees of freedom for ${v}\in \tilde{\mathcal{A}}_{3}\Lambda^{2}$ are defined by the moments on each element $K\in \mathcal{F}$:
$$
\int_{K}v\, dx.
$$

\begin{lemma}
The space $\tilde{\mathcal{A}}_{2}\Lambda^{1}$ with the degrees of freedom defined above is unisolvent. 
\end{lemma}
\begin{proof}
The local shape function space and the set of degrees of freedom have the same dimension. Therefore we only show that if all the degrees of freedom vanish on a function $\bm{u}\in \tilde{\mathcal{A}}_{2}\Lambda^{1}$, then $\bm{u}=0$. Actually, if all the degrees of freedom vanish, then we have $\displaystyle\int_{\partial K}\bm{u}\cdot\bm{\nu}_{\partial K}\, dx=0, \, \forall K\in \mathcal{F}$. By Stokes' theorem, this implies that $\displaystyle\int_{K}\div \bm{u}\, dx=0$ and further $\div \bm{u}=0$ since the image of $\div$ on $\tilde{\mathcal{A}}_{2}\Lambda^{1}$ is the space of constants. Then $\bm{u}$ can be expressed as $\bm{u}=\curl \phi$ by some potential $\phi\in \mathcal{A}_{3}\Lambda^{0}(K)$. Since $\bm{u}$ vanishes at vertices, $\phi$ have vanishing first order derivatives at vertices. Moreover, the function values of $\phi$ at the four vertices have to be equal, by the edge degrees of freedom and Stokes' theorem on edges. The potential $\phi$ is determined up to a constant, so we are allowed to choose this common vertex value to be zero. So $\phi=0$ and therefore $\bm{u}=0$.
\end{proof}

\begin{comment}
\begin{proof}
The local shape function space and the set of degrees of freedom have the same dimension. Therefore we only show that if all the degrees of freedom vanish on a function $\bm{u}\in \tilde{\mathcal{A}}_{2}\Lambda^{1}$, then $\bm{u}=0$.   Actually, we have $\tilde{\mathcal{A}}_{2}\Lambda^{1}(\mathcal{T}_{h}^{2})\subset \mathcal{A}_{2}\Lambda^{1}(\mathcal{T}_{h}^{2})$. By the same argument as the unisolvence of $\mathcal{A}_{1}\Lambda^{2}$, we see that 
$$
\bm{u}=c_{1}\left (
\begin{array}{c}
x(1-x)\\
 0
 \end{array}
  \right )+c_{2}\left ( 
  \begin{array}{c}
  0\\
  y(1-y)
  \end{array}
  \right ), 
$$
where $\left (x(1-x), 0 \right )^{T}$ and $ \left ( 0, y(1-y)\right )^{T}$ are the two bubbles of $\mathcal{A}_{1}\Lambda^{2}$ and $c_{1}$, $c_{2}$ are two constants. Since these two bubbles are not included in the reduced space $\tilde{\mathcal{A}}_{2}\Lambda^{1}$, we have $c_{1}=c_{2}=0$. This proves the unisolvence. 
\end{proof}
\end{comment}

We verify the following properties.
\begin{lemma}
The reduced spaces satisfy
\begin{itemize}
\item
$
\curl  \mathcal{A}_{3}\Lambda^{0}(\mathcal{T}_{h}^{2}) \subset \tilde{\mathcal{A}}_{2}\Lambda^{1}(\mathcal{T}_{h}^{2}) $,
\item $\div \tilde{\mathcal{A}}_{2}\Lambda^{1}(\mathcal{T}_{h}^{2})  \subset \tilde{\mathcal{A}}_{1}\Lambda^{2}(\mathcal{T}_{h}^{2})$,
\item
the complex
\begin{align}\label{complex:reduced}
\begin{diagram}
0 & \rTo &\mathbb{R} & \rTo & \mathcal{A}_{3}\Lambda^{0}(\mathcal{T}_{h}^{2}) & \rTo^{\curl} &\tilde{\mathcal{A}}_{2}\Lambda^{1}(\mathcal{T}_{h}^{2}) & \rTo^{\div} &\tilde{\mathcal{A}}_{1}\Lambda^{2}(\mathcal{T}_{h}^{2})& \rTo^{} & 0
\end{diagram}
\end{align}
is exact.
\end{itemize}
\end{lemma}
\begin{proof}
The inclusions are trivial by definition.  To show the exactness, we observe that if $\bm{u}\in \tilde{\mathcal{A}}_{2}\Lambda^{1}(\mathcal{T}_{h}^{2}) $ satisfies $\div\bm{u}=0$, then $\bm{u}$ has to be in $\curl \mathcal{A}_{3}\Lambda^{0}(\mathcal{T}_{h}^{2}) $ by 
the exactness of the Adini complex \eqref{2d-adini-complex} and the fact $\tilde{\mathcal{A}}_{2}\Lambda^{1}\subset {\mathcal{A}}_{2}\Lambda^{1}$.
\end{proof}

The local spaces of the reduced complex \eqref{complex:reduced} can also be constructed by the Poincar\'{e}/Koszul operators \cite{hiptmair1999canonical,Arnold.D;Falk.R;Winther.R.2006a,christiansen2016generalized}: 
$$
\tilde{\mathcal{A}}_{2}\Lambda^{1}(K)=\curl\tilde{\mathcal{A}}_{3}\Lambda^{0}(K)+\kappa \tilde{\mathcal{A}}_{1}\Lambda^{2}(K), \quad \forall K\in \mathcal{F},
$$
where $\kappa: v\mapsto v\bm{x}$ maps a scalar function to a vector valued function by multiplying $\bm{x}$ with a chosen origin. 
This also implies the local exactness and the approximation property that piecewise linear functions and piecewise constants are contained in the last two spaces respectively.  

Since the reduced space $\tilde{\mathcal{A}}_{2}\Lambda^{1}(\mathcal{T}_{h}^{2})$ is a subspace of ${\mathcal{A}}_{2}\Lambda^{1}(\mathcal{T}_{h}^{2})$ with vertex continuities, we have the discrete Korn inequality and the convergence as a non-conforming element.
\begin{theorem}
The reduced space $\tilde{\mathcal{A}}_{2}\Lambda^{1}(\mathcal{T}_{h}^{2})$ is $H(\div)$-conforming, satisfies the discrete Korn inequality stated in Theorem \ref{thm:korn} and leads to convergent schemes as a non-conforming element for Poisson type equations (as Theorem \ref{thm:convergence-poisson}). 
\end{theorem}

A similar diagram as \eqref{3d-diagram-serendipity} also exists for this reduced and DoF-transferred complex.

%%%%%%%%%%%%%%%%%%%%%%%%%%%%%%%%%%%%%%%%%%%%%%%%%%%

\section{Conclusions and future directions}\label{sec:conclusion}

The DoF-transfer operation proposed in this paper increases the regularity of shape functions at vertices, which reduces global DoFs and leads to convergent  non-conforming  elements. On the other hand, the serendipity operation reduces local DoFs. Both of these operations preserve the unisolvence and exactness of the discrete complex on contractible domains.
Based on the two operations and the serendipity finite element families on cubes \cite{arnold2014finite,GK2016}, we constructed several complexes with vertex DoFs and fewer local and global DoFs than standard tensor product elements. 

There is always a trade-off between nodal (nonstandard) and non-nodal (standard) elements. For $H(\div)$, hybridization is possible for the standard face elements, and these standard elements are better studied from the perspective of solvers and software. On the other hand, nodal elements have fewer global DoFs and sometimes nodal-type interpolations admit a canonical construction of bases with moderate condition numbers, sparsity and rotational symmetry (c.f. \cite{christiansen2016nodal}). In summary, the seemingly simple DoF-transfer operation changes the topology of the element connections, which in turn give new opportunities for the study of bases, solvers and applications. %These are relevant issues to be further explored.  

We have seen that the DoF-transfer operation on cubical meshes often renders  non-conforming  convergence. Therefore, in terms of  non-conforming  plate or Stokes elements, the nodal continuity seems natural and necessary when we require a discrete sequence that is both conforming as a de Rham complex and  non-conforming  (convergent) as a Stokes complex. In contrast, another well known complex of  non-conforming  elements, consisting of the Morley and the Crouzeix-Raviart elements, is not conforming as a de Rham complex~\cite{brenner2015forty,linke2018quasi}.

The DoF-transfer operation can also be applied to other element shapes. 
The discussions in \cite{christiansen2016nodal} for nodal elements are only focused on finite elements with complete polynomial shape spaces. The general approach proposed in this paper yields some new elements and complexes on simplicial meshes, at least for the family with regularity 1 in \cite{christiansen2016nodal}.  For example, we can modify a 2D Lagrange $\rightarrow$ Raviart-Thomas $\rightarrow$ DG complex by the DoF-transfer operation to obtain a Stenberg type element with Raviart-Thomas shape functions, which is continuous at vertices. This complex is also exact on contractible domains and from the viewpoint of differential complexes, we avoid the use of macroelement techniques.

On simplicial meshes, complexes starting with the modified Zienkiewicz element \cite{Wang.M;Shi.Z;Xu.J2007NM} may also be explored as an analogue of the discussions in this paper. 

Figure \ref{fig:adini2d-geometry} shows a combination of cubical and simplicial elements. The edge modes of both families are the same. This renders a possibility for flexible meshes and varying polynomial degrees in high order methods and adaptivity.

As explained in \cite{christiansen2016nodal}, the 1-forms in the 2D Hermite complex can be represented by nodal basis functions. For the standard serendipity element (0-forms in the Adini complex), nodal basis functions have been studied in \cite{FG2014}. Therefore the constructions in this paper could also shed some light on computational issues, especially for high order methods. We regard this respect as a further direction. 

The DoF-transfer operation is defined as changing vertex and edge degrees of freedom. In regard to the discussions in \cite{christiansen2016nodal}, this corresponds to the case $\ell=1$ where $\ell$ is a regularity index (denoted as $r=1$ in \cite{christiansen2016nodal}). More general DoF-transfer operations can be explored to cover the family $\ell=2$ and other global space reductions.

\section*{Acknowledgements}

AG was supported in part by National Science Foundation Award DMS-1522289.
KH was supported in part by the European Research Council under the European Union's Seventh Framework Programme (FP7/2007-2013) / ERC grant agreement 339643. 
SZ was supported in part by the National Natural Science Foundation of China with Grant No. 11471026.

\begin{figure}
\setlength{\unitlength}{1.2cm}
\begin{picture}(2,4)(-2,0)
\put(0,0){
\begin{picture}(2,2)
\put(-1, 0){\line(1,0){2}} 
\put(1, 2){\line(-1,0){2}}
\put(-1,0){\line(0,1){2}}
\put(1, 2){\line(0, -1){2}}
\put(-1,0){\circle*{0.1}}
\put(-1,0){\circle{0.2}}
\put(1.,0){\circle*{0.1}}
\put(1.,0){\circle{0.2}}
\put(-1,2){\circle*{0.1}}
\put(-1,2){\circle{0.2}}
\put(1,2){\circle*{0.1}}
\put(1,2){\circle{0.2}}
\put(0,4){\circle*{0.1}}
\put(0,2.5){\circle*{0.1}}
\put(0,4){\circle{0.2}}
\put(-1, 2){\line(1,2){1}}
\put(1, 2){\line(-1,2){1}}
\end{picture}
}

\put(1.5, 1){\vector(1, 0){1}}
\put(1.68, 1.15){$\curl$}

\put(4,0){
\put(-1, 0){\line(1,0){2}} 
\put(1, 2){\line(-1,0){2}}
\put(-1,0){\line(0,1){2}}
\put(1, 2){\line(0, -1){2}}
\put(-0.95,0){\circle*{0.1}}
\put(-1.05,0){\circle*{0.1}}
\put(1.05,0){\circle*{0.1}}
\put(0.95,0){\circle*{0.1}}
\put(-0.95,2){\circle*{0.1}}
\put(-1.05,2){\circle*{0.1}}
\put(1.05,2){\circle*{0.1}}
\put(0.95,2){\circle*{0.1}}
\put(0.05,4){\circle*{0.1}}
\put(-0.05,4){\circle*{0.1}}
\put(0, 2){\vector(0,1){0.3}}
\put(0, 0){\vector(0,-1){0.3}}
\put(-1, 1){\vector(-1,0){0.3}}
\put(1, 1){\vector(1,0){0.3}}
\put(-0.5, 3){\vector(-1,1){0.3}}
\put(0.5, 3){\vector(1,1){0.3}}
\put(-1, 2){\line(1,2){1}}
\put(1, 2){\line(-1,2){1}}
\put(-0.3,2.5){+3}
%\put(1, 1){\circle*{0.1}}
%\put(-0.2, 1){\circle*{0.1}}
%\put(0.2, 1){\circle*{0.1}}
\put(-0.3, 0.9){+2}
}

\put(5.5, 1){\vector(1, 0){1}}
\put(5.75, 1.1){{$\div$}}

\put(8,0){
\begin{picture}(2,2)
\put(-1, 0){\line(1,0){2}} 
\put(1, 2){\line(-1,0){2}}
\put(-1,0){\line(0,1){2}}
\put(1, 2){\line(0, -1){2}}
\put(-1, 2){\line(1,2){1}}
\put(1, 2){\line(-1,2){1}}
\put(-0.3,2.5){+3}
\put(-0.3, 0.9){+3}
\end{picture}
}

%\put(9, 1){\vector(1, 0){1}}
\end{picture}
\caption{ Adini-Hermite-de Rham complex on flexible geometries}
\label{fig:adini2d-geometry}
\end{figure}
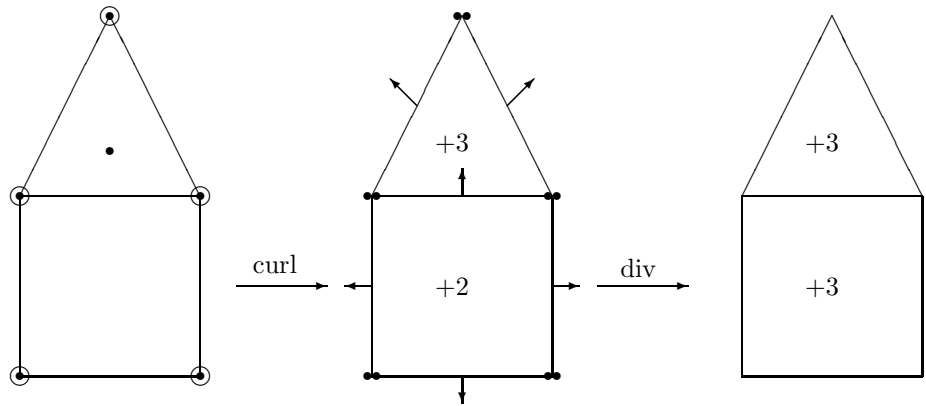

%%%%%%%%%%%%%%%%%
\bibliographystyle{siam}      % mathematics and physical sciences
\bibliography{cube}{}   % name your BibTeX data base

\end{document}